\documentclass[onefignum,onetabnum]{siamart171218}

\usepackage{amsfonts}
\usepackage{amssymb}
\usepackage{graphicx}
\usepackage{epstopdf}
\usepackage{algorithmic}
\ifpdf
  \DeclareGraphicsExtensions{.eps,.pdf,.png,.jpg}
\else
  \DeclareGraphicsExtensions{.eps}
\fi

\usepackage{bm}
\usepackage{tikz-cd}
\usepackage{subfig}
\usepackage{mathtools}
\usepackage{stmaryrd}
\newsiamremark{remark}{Remark}
\newsiamremark{assumption}{Assumption}

\newsiamremark{hypothesis}{Hypothesis}
\crefname{hypothesis}{Hypothesis}{Hypotheses}
\newsiamthm{claim}{Claim}

\def\XXint#1#2#3{{\setbox0=\hbox{$#1{#2#3}{\int}$ }
\vcenter{\hbox{$#2#3$ }}\kern-.6\wd0}}

\newcommand{\tr}{{\rm tr}}

\headers{$C^0$ (non-Langrange) FEM for PDEs with Cordes Coefficients}{S. Wu}

\title{$C^0$ Finite Element Approximations of Linear Elliptic
Equations in Non-divergence Form and Hamilton-Jacobi-Bellman Equations
with Cordes Coefficients
\thanks{%Submitted to the editors DATE.  
The work of Shuonan Wu is supported in part by the National Natural
Science Foundation of China grant No.  11901016 and the startup grant
from Peking University. }
}

\author{
Shuonan Wu\thanks{School of Mathematical Sciences,
  Peking University, Beijing 100871, China 
(\email{snwu@math.pku.edu.cn}, \url{http://dsec.pku.edu.cn/\~snwu}).}
}

\usepackage{amsopn}

\ifpdf
\hypersetup{
  pdftitle={$C^0$ (non-Lagrange) FEM for PDEs with Cordes Coefficients},
  pdfauthor={S. Wu}
}
\fi

\begin{document}

\maketitle

% REQUIRED
\begin{abstract}
  This paper is concerned with $C^0$ (non-Lagrange) finite element
  approximations of the linear elliptic equations in non-divergence
  form and the Hamilton-Jacobi-Bellman (HJB) equations with Cordes
  coefficients.  Motivated by the Miranda-Talenti estimate, a discrete
  analog is proved once the finite element space is $C^0$ on the
  $(n-1)$-dimensional subsimplex (face) and $C^1$ on
  $(n-2)$-dimensional subsimplex.  The main novelty of the
  non-standard finite element methods is to introduce an interior
  penalty term to argument the PDE-induced variational form of the
  linear elliptic equations in non-divergence form or the HJB
  equations. As a distinctive feature of the proposed methods, no
  penalization or stabilization parameter is involved in the
  variational forms. As a consequence, the coercivity constant (resp.
  monotonicity constant) for the linear elliptic equations in
  non-divergence form (resp. the HJB equations) at discrete level is
  exactly the same as that from PDE theory.  The quasi-optimal order
  error estimates as well as the convergence of the semismooth Newton
  method are established.  Numerical experiments are provided to
  validate the convergence theory and to illustrate the accuracy and
  computational efficiency of the proposed methods.
\end{abstract}

% REQUIRED

% REQUIRED

\section{Introduction}
In this paper, we study the $C^0$ (non-Lagrange) finite element
approximations of the linear elliptic equations in non-divergence form 
\begin{equation} \label{eq:nondiv}
Lu := A: D^2 u = f \quad \text{in }\Omega, \qquad  
u = 0 \quad \text{on }\partial \Omega,
\end{equation}
and the Hamilton-Jacobi-Bellman (HJB) equations 
\begin{equation} \label{eq:HJB}
\sup_{\alpha \in \Lambda} (L^\alpha u - f^\alpha) = 0 
\quad \text{in }\Omega, \qquad 
u = 0 \quad \text{on }\partial\Omega,
\end{equation}
where $\Lambda$ is a compact metric space, and  
$$ 
L^\alpha v := A^\alpha : D^2 v + \bm{b}^\alpha \cdot \nabla
v - c^\alpha v.
$$ 
Here, $D^2u$ and $\nabla u$ denote the Hessian and gradient of
real-valued function $u$, respectively. $\Omega$ is a bounded, open,
convex polytope in $\mathbb{R}^n~(n=2,3)$.  Problems \eqref{eq:nondiv}
and \eqref{eq:HJB} arise in many applications from areas such as
probability and stochastic processes. In particular, the HJB equations
\eqref{eq:HJB}, which are of the class of fully nonlinear second order
partial differential equations (PDEs), play a fundamental role in the
field of stochastic optimal control \cite{fleming2006controlled}. In
the study of fully nonlinear second order PDEs, linearization
techniques naturally lead to problems such as \eqref{eq:nondiv}
\cite{caffarelli1997properties}.

%% Results without Cordes condition
In contract to the PDEs in divergence form, the theory of linear
elliptic equations in non-divergence form and, more generally, fully
nonlinear PDEs, hinges on different solution concepts such as strong
solutions, viscosity solutions, or $H^2$ solutions. For these
different solution concepts, numerical methods for \eqref{eq:nondiv}
and \eqref{eq:HJB} have experienced some rapid developments in recent
years. In \cite{feng2017finite, feng2018interior}, the authors proposed
non-standard primal finite element methods for \eqref{eq:nondiv} with
the coefficient matrix $A \in C^0(\bar{\Omega}; \mathbb{R}^{n\times
n})$. They showed the convergence in the sense of $W^{2,p}$ strong
solutions by establishing the discrete Calderon-Zygmund estimates. In
\cite{nochetto2018discrete}, the authors studied a two-scale method
based on the integro-differential formulation of \eqref{eq:nondiv},
where the monotonicity under the weakly acute mesh condition and
discrete Alexandroff-Bakelman-Pucci estimates were established to
obtain the pointwise error estimates.  In the Barles-Souganidis
framework \cite{barles1991convergence}, the monotonicity is a key
concept of numerical schemes for the convergence to the viscosity
solutions of fully nonlinear PDEs, which is also applicable to the
monotone finite difference methods (FDM) for the HJB equations
\cite{bonnans2003consistency, fleming2006controlled}. A monotone
finite element like scheme for the HJB equations was proposed in
\cite{camilli2009finite}.  The convergent monotone finite element
methods (FEM) for the viscosity solutions of parabolic HJB equations
were proposed and analyzed in \cite{jensen2013convergence,
jensen2017finite}. For a general overview, we refer the reader to the
survey articles \cite{feng2013recent, neilan2017numerical}.  Recently
a narrow-stencil FDM for the HJB equations based on the generalized
monotonicity was proposed in \cite{feng2019narrow}, where the
numerical solution was proved to be convergent to the viscosity
solution. 

%% Variation settng, with Cordes condition 
For the PDE theory of $H^2$ solutions, it is allowable to have the
discontinuous coefficient matrix $A$ in \eqref{eq:nondiv}.  As a compensation, 
the coefficients are required to satisfy the {\it Cordes
condition} stated below in Definition \ref{df:nondiv-Cordes} for
\eqref{eq:nondiv} and Definition \ref{df:HJB-Cordes} for
\eqref{eq:HJB}, respectively. We refer the reader to
\cite{maugeri2000elliptic} for the analysis of PDEs with discontinuous
coefficients under the Cordes condition. In addition, the analysis of
the problems \eqref{eq:nondiv} and \eqref{eq:HJB} hinges on the
following Miranda-Talenti estimate. 
\begin{lemma}[Miranda-Talenti estimate] \label{lm:MT}
Suppose that $\Omega \subset
\mathbb{R}^n$ is a bounded convex domain. Then, for any $v \in
H^2(\Omega) \cap H_0^1(\Omega)$, 
\begin{equation} \label{eq:MT}
|v|_{H^2(\Omega)} \leq \|\Delta v\|_{L^2(\Omega)}.
\end{equation}
\end{lemma}

%% Comments on the finite element with conforming and nonconforming
For the finite element approximations of $H^2$ solutions, 
the most straightforward way is to apply the $C^1$-conforming finite
elements \cite{ciarlet1978finite}, which are sometimes considered
impractical on unstructured meshes (at least $\mathcal{P}_5$ in 2D and
$\mathcal{P}_9$ in 3D).  It is worth mentioning that the $H^2$
nonconforming elements would fail to mimic the Miranda-Talenti
estimate at the discrete level. For example, a direct calculation
shows that three basis functions of the Morley element
\cite{ciarlet1978finite} are harmonic.  

%% Existing method: penalty or stabilized method 
%% 1. DG, C0-DG
Instead of the $C^1$-conforming finite element methods, the first
discontinuous Galerkin (DG) method for \eqref{eq:nondiv} in the case
of discontinuous coefficients with Cordes condition was proposed
in \cite{smears2013discontinuous}, which has been extended to the
elliptic and parabolic HJB equations in \cite{smears2014discontinuous,
smears2016discontinuous}. In \cite{neilan2019discrete}, the authors
developed the $C^0$-interior penalty DG methods to both
\eqref{eq:nondiv} and \eqref{eq:HJB}. The above DG methods are
applicable when choosing the penalization parameters suitably
large. A mixed method for \eqref{eq:nondiv} based on the stable finite
element Stokes spaces was proposed in \cite{gallistl2017variational},
which has been extended to the HJB equations in
\cite{gallistl2019mixed}. Other numerical methods for
\eqref{eq:nondiv} include the discrete Hessian method
\cite{lakkis2011finite}, weak Galerkin method \cite{wang2018primal},
and least square methods \cite{li2019sequential, qiu2019adaptive}.

%% Paragraphs on the novelty of the current work.
The primary goal of this paper is to develop and analyze the $C^0$
primal finite element approximations of \eqref{eq:nondiv} and
\eqref{eq:HJB} without introducing any penalization parameter.  In
view of the proof of Miranda-Talenti estimate
\cite{maugeri2000elliptic}, the difference between $\|\Delta
u\|_{L^2(\Omega)}^2$ and $|u|_{H^2(\Omega)}^2$ is a positive term
that involves the mean curvature of $\partial \Omega$.  The key idea
in this paper is to adopt the $C^0$ finite element with the enhanced
regularity on some subsimplex so as to have the ability to detect the
information of mean curvature. More precisely, we show in Section
\ref{sec:FEM} that, a feasible condition pertaining to the finite
element is the $C^1$-continuity on the $(n-2)$-dimensional subsimplex.
In 2D case, a typical family of finite elements that meets this
requirement is the family of $\mathcal{P}_k$-Hermite finite elements
$(k\geq 3)$ depicted in Figure \ref{fig:Hermite-2Dk3}. In 3D case, the
family of Argyris finite elements with local space $\mathcal{P}_k (k
\geq 5)$ \cite{neilan2015discrete, christiansen2018nodal} satisfies
the condition.  

Having the families of 2D Hermite finite elements and 3D Argyris
finite elements, we prove a discrete analog of the Miranda-Talenti
estimate in Lemma \ref{lm:Hermite-MT}, on any conforming triangulation
of convex polytope $\Omega$. The jump terms in the discrete
Miranda-Talenti-type estimate \eqref{eq:Hermite-MT} naturally induce
the interior penalty in the variational form of the linear elliptic
equations in non-divergence form \eqref{eq:nondiv-bilinear} or the HJB
equations \eqref{eq:HJB-form}. However, in the same sprite of
\cite{feng2017finite}, the proposed methods are not the DG methods per
se since no penalization parameter is used. The convergence analysis
mimics the analysis of $H^2$ solutions to a great extent. As a
striking feature of the proposed methods, the coercivity constant
(resp. monotonicity constant) for the linear elliptic equations in
non-divergence form (resp. the HJB equations) at discrete level is
exactly the same as that from PDE theory. Since the methods are
consistent, the coercivity or monotonicity naturally leads to the
energy norm error estimates. 

%% Organization 
The rest of the paper is organized as follows. In Section
\ref{sec:pre}, we establish the notation and state some preliminaries
results.  In Section \ref{sec:FEM}, we state the finite element spaces
on which the discrete Miranda-Talenti-type estimate holds.  We then
give the applications to the discretizations of the linear elliptic
equations in non-divergence form \eqref{eq:nondiv} and the HJB equations
\eqref{eq:HJB}, respectively, in Section \ref{sec:nondiv} and Section
\ref{sec:HJB}.  Numerical experiments are presented in Section
\ref{sec:numerical}. 

For convenience, we use $C$ to denote a generic positive
constant that may stand for different values at its different
occurrences but is independent of the mesh size $h$. The notation $X
\lesssim Y$ means $X \leq CY$. 

\section{Preliminaries} \label{sec:pre}

In this section, we first review the $H^2$ solutions to the linear
elliptic equations in non-divergence form and the HJB equations under the
Cordes conditions. Let $\Omega$ be a bounded, open, convex domain in
$\mathbb{R}^n~(n=2,3)$, in which the Miranda-Talenti estimate
\eqref{eq:MT} holds. We shall use $U$ to denote a generic subdomain of
$\Omega$ and $\partial U$ denotes its boundary. Given an integer $k
\geq 0$, let $H^k(U)$ and $H_0^k(U)$ denote the usual Sobolev spaces.
We also denote $V := H^2(\Omega) \cap H_0^1(\Omega)$. 

\subsection{Review of strong solutions to the linear elliptic equations in
non-divergence form}
For the problem \eqref{eq:nondiv}, it is assumed that $A\in
L^\infty(\Omega;\mathbb{R}^{n\times n})$, and that $L$ is uniformly
elliptic, i.e., there exists $\underline{\nu}, \bar{\nu} > 0$ such
that 
\begin{equation} \label{eq:nondiv-elliptic}
\underline{\nu}|\bm{\xi}|^2 \leq \bm{\xi}^tA(x)\bm{\xi} \leq \bar{\nu}
|\bm{\xi}|^2 \qquad \forall \bm{\xi} \in \mathbb{R}^n, \text{ a.e. in
}\Omega.
\end{equation} 
It is well-known that, the above assumptions can not guarantee the
well-posedness of \eqref{eq:nondiv}; See, for instance, the example in
\cite[p. 185]{gilbarg2015elliptic}. In addition, we assume the
coefficients satisfy the Cordes condition below. 

\begin{definition}[Cordes condition for \eqref{eq:nondiv}]
\label{df:nondiv-Cordes}
The coefficient satisfies that there is an $\varepsilon \in (0, 1]$
such that 
\begin{equation} \label{eq:nondiv-Cordes}
\frac{|A|^2}{(\tr A)^2} \leq \frac{1}{n-1 + \varepsilon} 
\qquad \text{a.e. in }\Omega.
\end{equation}
\end{definition}
In two dimensions, it is not hard to show that uniform ellipticity
implies the Cordes condition \eqref{eq:nondiv-Cordes} with
$\varepsilon = 2\underline{\nu} / (\underline{\nu} + \bar{\nu})$, see
\cite{smears2014discontinuous}.  Define the strictly positive function
$\gamma \in L^\infty(\Omega)$ by 
\begin{equation} \label{eq:nondiv-gamma}
\gamma := \frac{\tr A}{|A|^2}.
\end{equation}
Hence, the variational formulation of \eqref{eq:nondiv} reads 
\begin{equation} \label{eq:nondiv-B}
B_0(u, v) = \int_\Omega \gamma f \Delta v\,\mathrm{d}x \quad \forall v \in
V,
\end{equation}
where the bilinear form $B_0: V \times V \to \mathbb{R}$ is defined by
\begin{equation}  \label{eq:nondiv-B-def}
B_0(w,v):= \int_\Omega \gamma Lw \Delta v\,\mathrm{d}x \quad \forall w, v \in
V.
\end{equation}
The well-posedness hinges on the following lemma; See \cite[Lemma
1]{smears2013discontinuous} for a proof.

\begin{lemma} \label{lm:nondiv-Cordes-prop}
Under the Cordes condition \eqref{eq:nondiv-Cordes}, for any $v\in
H^2(\Omega)$ and open set $U \subset \Omega$, the following inequality
holds a.e. in $U$: 
\begin{equation} \label{eq:nondiv-Cordes-prop}
|\gamma L v - \Delta v| \leq \sqrt{1-\varepsilon} |D^2v|.
\end{equation}
\end{lemma}

Using the Miranda-Talenti estimate in Lemma \ref{lm:MT}, Lemma
\ref{lm:nondiv-Cordes-prop} and the Lax-Milgram Lemma, it is readily
seen that there exists a unique strong solution
to \eqref{eq:nondiv-B} under the Cordes condition
\eqref{eq:nondiv-Cordes}. The proof is given in \cite[p.
256]{maugeri2000elliptic}, see also \cite{smears2013discontinuous,
neilan2019discrete}. 
%%We note that the analysis of propose finite
%%element method in Section \ref{sc:nondiv} exactly follows exactly
%%the procedure of the PDE level.  

%% HJB %%
\subsection{Review of strong solutions to the HJB equation}
For the HJB equations \eqref{eq:HJB}, the coefficient $A^\alpha \in
C(\bar{\Omega} \times \Lambda;\mathbb{R}^{n\times n})$ is assumed
uniformly elliptic, i.e., there exists $\underline{\nu}, \bar{\nu} >
0$ such that 
\begin{equation} \label{eq:HJB-elliptic}
\underline{\nu}|\bm{\xi}|^2 \leq \bm{\xi}^tA^\alpha(x)\bm{\xi} \leq \bar{\nu}
|\bm{\xi}|^2 \qquad \forall \bm{\xi} \in \mathbb{R}^n, \text{ a.e. in
}\Omega, ~\forall \alpha \in \Lambda.
\end{equation} 
The corresponding Cordes condition for $A^\alpha$, $\bm{b}^\alpha \in
C(\bar{\Omega} \times \Lambda;\mathbb{R}^n)$ and $c^\alpha \in
C(\bar{\Omega} \times \Lambda)$ can be stated as follows.  
\begin{definition}[Cordes condition for \eqref{eq:HJB}]
\label{df:HJB-Cordes}
The coefficients satisfy 
\begin{subequations} \label{eq:HJB-Cordes}
\begin{enumerate}
\item whenever $\bm{b}^\alpha \not\equiv \bm{0}$ or $c^\alpha
\not\equiv 0$ for some $\alpha \in \Lambda$, there exist $\lambda > 0$
and $\varepsilon\in (0,1]$ such that 
\begin{equation} \label{eq:HJB-Cordes1}
\frac{|A^\alpha|^2 + |\bm{b}^\alpha|^2/(2\lambda) +
(c^\alpha/\lambda)^2}{(\tr A^\alpha + c^\alpha/\lambda)^2} \leq \frac{1}{n +
\varepsilon} \qquad \text{a.e. in
}\Omega, ~\forall \alpha \in \Lambda.
\end{equation}
\item whenever $\bm{b}^\alpha\equiv \bm{0}$ and $\bm{c}^\alpha \equiv
0$ for all $\alpha \in \Lambda$, there is an $\varepsilon \in (0, 1]$
such that 
\begin{equation} \label{eq:HJB-Cordes2}
\frac{|A^\alpha|^2}{(\tr A^\alpha)^2} \leq \frac{1}{n-1 + \varepsilon}
\qquad \text{a.e. in }\Omega.
\end{equation}
\end{enumerate}
\end{subequations}
\end{definition}
Similar to \eqref{eq:nondiv-gamma} for the linear elliptic equations
in non-divergence form, for each $\alpha \in \Lambda$, we define 
\begin{equation} \label{eq:HJB-gamma}
\gamma^\alpha := \left\{
\begin{array}{cl}
\displaystyle\frac{\tr A^\alpha + c^\alpha/\lambda}{|A^\alpha|^2 +
|\bm{b}^\alpha|^2/(2\lambda) +
(c^\alpha/\lambda)^2} & ~~\bm{b}^\alpha \not\equiv \bm{0} \text{
or } c^\alpha \not\equiv 0 ~ \text{ for some }\alpha \in \Lambda,
\\
\displaystyle\frac{\tr A^\alpha}{|A^\alpha|^2} & ~~\bm{b}^\alpha
\equiv \bm{0} \text{ and } c^\alpha \equiv 0 ~ \text{ for all }
\alpha \in \Lambda.
\end{array}
\right.
\end{equation}
Note here that the continuity of data implies $\gamma^\alpha \in
C(\bar{\Omega}\times \Lambda)$. Define the operator $F_\gamma:
H^2(\Omega) \to L^2(\Omega)$ by 
\begin{equation} \label{eq:HJB-F}
F_\gamma[v] := \sup_{\alpha \in \Lambda} \gamma^\alpha(L^\alpha v -
f^\alpha).
\end{equation}
It is readily seen that the HJB equations \eqref{eq:HJB} is in fact
equivalent to the problem $F_\gamma[u] = 0$ in $\Omega$, $u = 0$ on
$\partial \Omega$. The Cordes condition \eqref{eq:HJB-Cordes} leads to
the following lemma; See \cite[Lemma 1]{smears2014discontinuous} for a
proof.

\begin{lemma} \label{lm:HJB-Cordes-prop}
Under the Cordes condition \eqref{eq:HJB-Cordes}, for any open set $U
\subset \Omega$ and $w, v \in H^2(U)$, $z:= w-v$, the following
inequality holds a.e. in $U$:
\begin{equation} \label{eq:HJB-Cordes-prop}
|F_\gamma[w] - F_\gamma[v] - L_\lambda z| \leq \sqrt{1-\varepsilon}
\sqrt{|D^2 z|^2 + 2\lambda|\nabla z|^2 + \lambda^2|z|^2}.
\end{equation}
\end{lemma}

For $\lambda$ as in \eqref{eq:HJB-Cordes1}, we define a linear operator
$L_\lambda$ by 
\begin{equation} \label{eq:L-lambda}
L_\lambda v := \Delta v - \lambda v \qquad v \in V.
\end{equation}
Let the operator $M: V \to V^*$ be
\begin{equation} \label{eq:HJB-M}
\langle M[w], v \rangle := \int_\Omega F_\gamma[w] L_\lambda v
\,\mathrm{d}x \qquad \forall w, v \in V.
\end{equation}
where $\langle \cdot, \cdot \rangle$ is the dual pairing between $V^*$
and $V$.  Under the condition of Miranda-Talenti estimate in Lemma \ref{lm:MT}, it is
straightforward to show that 
\begin{equation} \label{eq:L-lambda-bound} 
\|L_\lambda v\|_{L^2(\Omega)}^2 \geq \int_{\Omega} |D^2v|^2 +
2\lambda|\nabla v|^2 + \lambda^2|v|^2\, \mathrm{d}x \qquad \forall v \in
V.
\end{equation} 
By using the Cordes condition \eqref{eq:HJB-Cordes}, one can show the
strong monotonicity of $M$. Together with the Lipschitz continuity of
$M$ by the compactness of $\Lambda$ and the Browder-Minty Theorem
\cite[Theorem 10.49]{renardy2006introduction}, one can show the
existence and uniqueness of the following problem: Find $u \in V$ such
that  
\begin{equation} \label{eq:HJB-var}
\langle M[u], v \rangle = 0 \qquad \forall v \in V.
\end{equation}
We refer to \cite[Theorem 3]{smears2014discontinuous} for a detailed
proof. 

\section{Finite element spaces and discrete Miranda-Talenti-type estimate}
\label{sec:FEM}

Let $\mathcal{T}_h$ be a conforming and shape-regular simplicial
triangulation of $\Omega$ and $\mathcal{F}_h$ be the set of all
faces of $\mathcal{T}_h$. Let $\mathcal{F}_h^i := \mathcal{F}_h
\setminus \partial \Omega$ and $\mathcal{F}_h^\partial := \mathcal{F}_h
\cap \partial\Omega$.  Here, $h := \max_{K\in \mathcal{T}_h} h_K$,
and $h_K$ is the diameter of $K$ (cf. \cite{ciarlet1978finite,
brenner2007mathematical}). Since each element has piecewise flat
boundary, the faces may also be chosen to be flat.  

We define the jump of a vector function $\bm{w}$ on an interior
 face $F = \partial K^+ \cap \partial K^-$ as follows:
$$ 
\llbracket \bm{w} \rrbracket|_F := \bm{w}^+ \cdot \bm{n}^+|_F + 
\bm{w}^- \cdot \bm{n}^-|_F,
$$ 
where $\bm{w}^\pm = \bm{w}|_{K^\pm}$ and $\bm{n}^\pm$ is the outward
unit normal vector of $K^\pm$, respectively. We also denote $\omega_F:= K^+ \cup K^-$
for any $F\in \mathcal{F}_h^i$. For an element $K \in \mathcal{T}_h$,
$(\cdot, \cdot)_K$ denotes the standard inner product on $L^2(K)$. The
standard inner products $\langle \cdot, \cdot \rangle_{\partial K}$
and $\langle \cdot, \cdot \rangle_F$, are defined in a similar
way.
%For simplicity of exposition, we use the convention $(\cdot, \cdot)_{\mathcal{T}_h} 
%:= \sum_{K \in \mathcal{T}_h}(\cdot, \cdot)_K$

For $F \in \mathcal{F}_h$, following \cite{smears2013discontinuous,
smears2014discontinuous}, we define the tangential gradient
$\nabla_T: H^s(F) \to \bm{H}_T^{s-1}(F)^n$ and the tangential
divergence $\mathrm{div}_T: \bm{H}_T^s(F) \to H^{s-1}(F)$, where $s
\geq 1$. Here, $\bm{H}_T^s(F):= \{\bm{v} \in H^s(F)^n:~\bm{v}\cdot
\bm{n}_F = 0~\text{on }F\}$. Let $\{\bm{t}_i\}_{i=1}^{n-1}$ be an
orthogonal coordinate system on $F$. Then, for $w \in H^s(F)$ and
$\bm{v} = \sum_{i=1}^{n-1} v_i \bm{t}_i$ with $v_i \in H^s(F)$ for
$i=1,\cdots, n-1$, define 
\begin{equation}
\nabla_T w := \sum_{i=1}^{n-1} \bm{t}_i \frac{\partial w}{\partial
  \bm{t}_i}, \qquad \mathrm{div}_T \bm{v} := \sum_{i=1}^{n-1}
  \frac{\partial v_i}{\partial \bm{t}_i}.
\end{equation}
We also define $\Delta_T w := \mathrm{div}_T\nabla_T w$ for $w \in
H^s(F)$, where $s \geq 2$. 

\subsection{The families of 2D Hermite finite elements and 3D Argyris
finite elements} 
In this subsection, we shall describe the finite elements that will be
used to solve the linear elliptic equations in non-divergence form
\eqref{eq:nondiv} and the HJB equations \eqref{eq:HJB}. More
precisely, we adopt the Hermite elements in 2D and the Argyris
elements in 3D \cite{neilan2015discrete, christiansen2018nodal}, both
of which have the $C^0$-continuity on the face (namely,
$(n-1)$-dimensional subsimplex) and $C^1$-continuity on the
$(n-2)$-dimensional subsimplex. 

\paragraph{The family of 2D Hermite finite elements} Following the
description of \cite{ciarlet1978finite, brenner2007mathematical}, the
geometric shape of Hermite elements is triangle $K$. The shape
function space is given as $\mathcal{P}_k(K)~(k\geq 3)$, where
$\mathcal{P}_k(K)$ denotes the set of polynomials with total degree
not exceeding $k$ on $K$. In $K$ the degrees of freedom are defined as
follows (cf. Figure \ref{fig:Hermite-2D}):
\begin{itemize}
\item Function value $v(\bm{a})$ and first order derivatives
$\partial_i v(\bm{a})$, $i=1,2$ at each vertex;
\item Moments $\int_e v q\,\mathrm{d}s, \forall q \in
\mathcal{P}_{k-4}(e)$ on each edge $e$;
\item Moments $\int_K v q\,\mathrm{d}x, \forall q \in
\mathcal{P}_{k-3}(K)$ on element $K$.
\end{itemize}
It is simple to check that the degrees of freedom given above form a
unisolvent set of $\mathcal{P}_k(K)$ for $k \geq 3$
\cite{ciarlet1978finite}.

\begin{figure}[!htbp]
\centering 
\captionsetup{justification=centering}
\subfloat[2D Hermite element, $k=3$]{
  \includegraphics[width=0.24\textwidth]{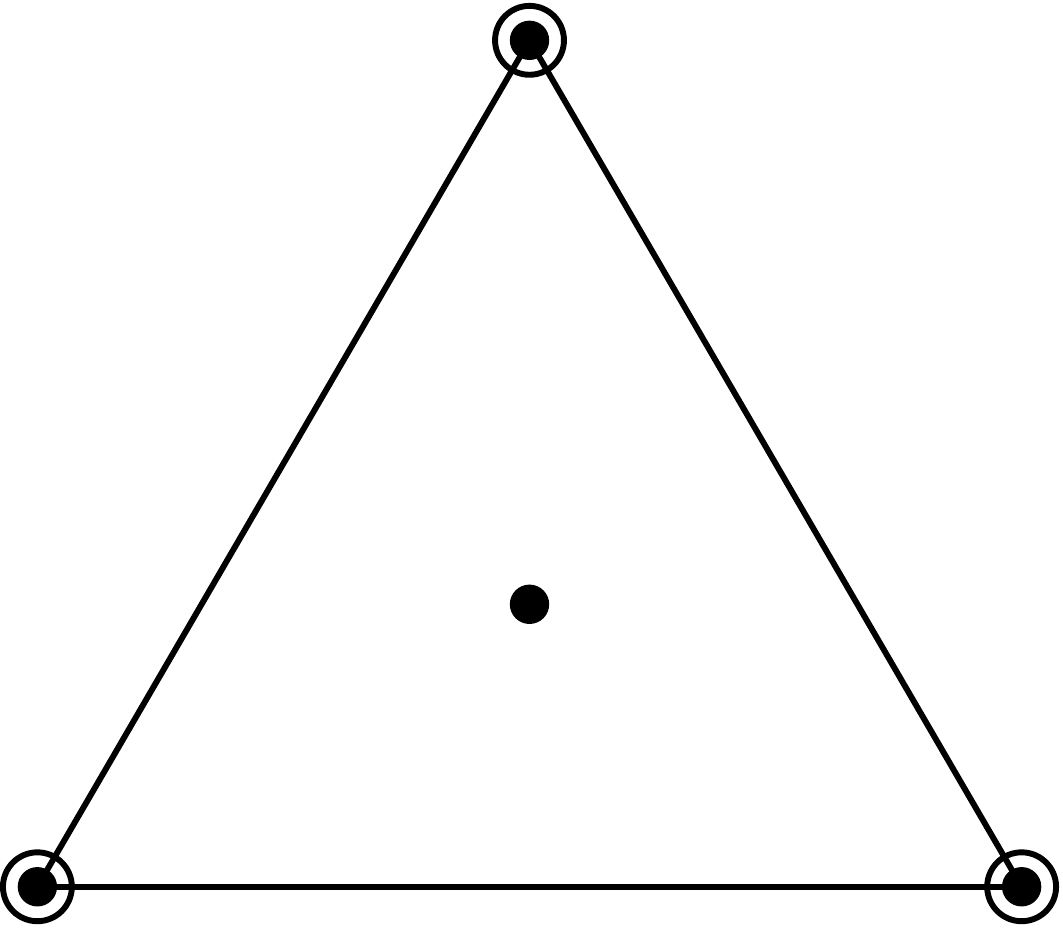}
  \label{fig:Hermite-2Dk3}
}\qquad %
\subfloat[2D Hermite element, $k=4$]{
  \includegraphics[width=0.24\textwidth]{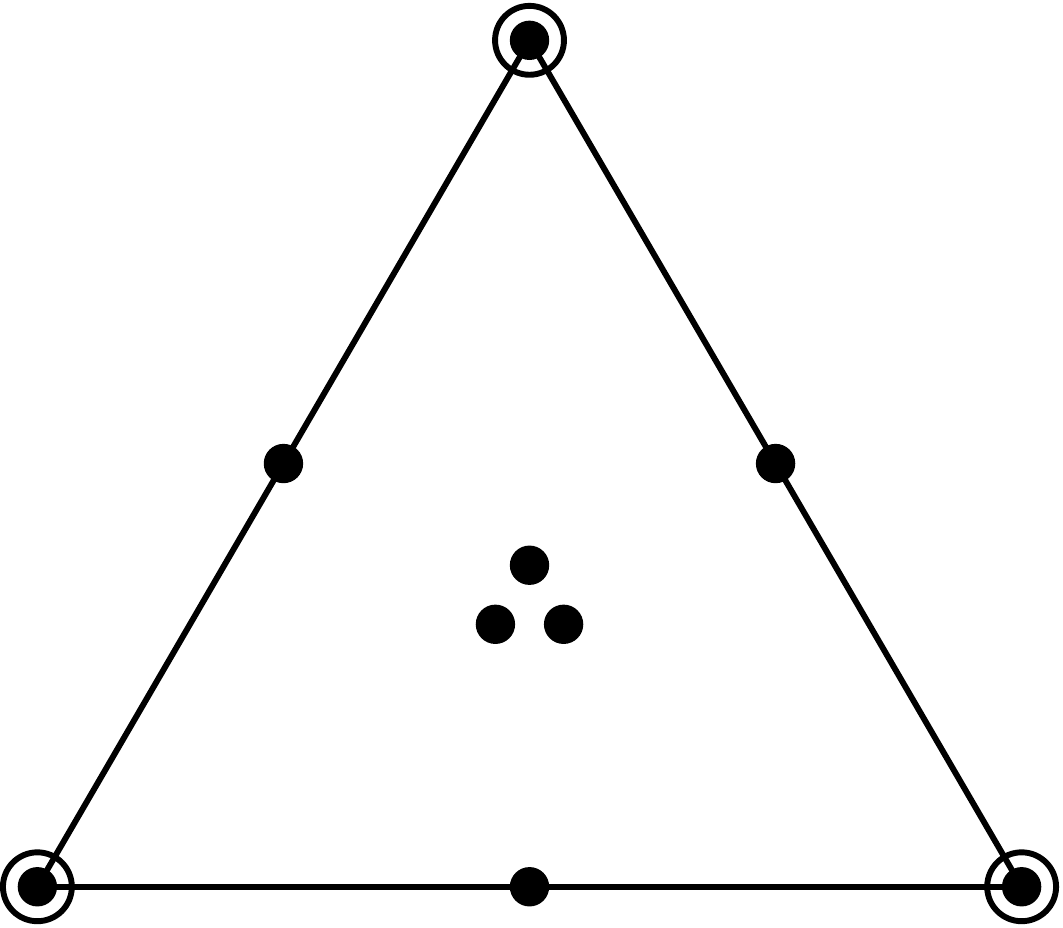}
  \label{fig:Hermite-2Dk4}
} 
\caption{Degrees of freedom of 2D $\mathcal{P}_k$ Hermite elements, in
  the case of $k=3$ and $k=4$}
\label{fig:Hermite-2D}
\end{figure}

\paragraph{The family of 3D Argyris finite elements} In 3D case, the
finite elements are required to be $C^0$ on face and $C^1$ on edge.
The typical elements that meet the requirement in 3D are the Argyris
elements \cite{christiansen2018nodal}, which coincide with each
component of the velocity finite elements in the 3D smooth de Rham
complex \cite{neilan2015discrete}. Given a  tetrahedron $K$, the shape
function space is given as $\mathcal{P}_k(K)$, for $k \geq 5$. In $K$
the degrees of freedom are defined as follows (cf. Figure
\ref{fig:Argyris-3D}):
\begin{itemize}
\item One function value and (nine) derivatives up to second
order at each vertex; 
\item Moments $\int_e v q \,\mathrm{d}s, \forall q \in
\mathcal{P}_{k-6}(e)$ on each edge $e$;
\item Moments $\int_e \frac{\partial v}{\partial \bm{n}_{e,i}} q
\,\mathrm{d}s, \forall q \in \mathcal{P}_{k-5}(e)$, $i=1,2$ on each edge $e$;
\item Moments $\int_F v q \,\mathrm{d}s, \forall q \in
\mathcal{P}_{k-6}(F)$ on each face $F$;
\item Moments $\int_T v q \,\mathrm{d}x, \forall q \in
\mathcal{P}_{k-4}(K)$ on element $K$.
\end{itemize}
Here, $\bm{n}_{e,i}~(i=1,2)$ are two unit orthogonal normal vectors that are orthogonal to the edge $e$.

\begin{figure}[!htbp]
\centering 
\captionsetup{justification=centering}
\subfloat[3D Argyris elements, $k=5$]{
  \includegraphics[width=0.28\textwidth]{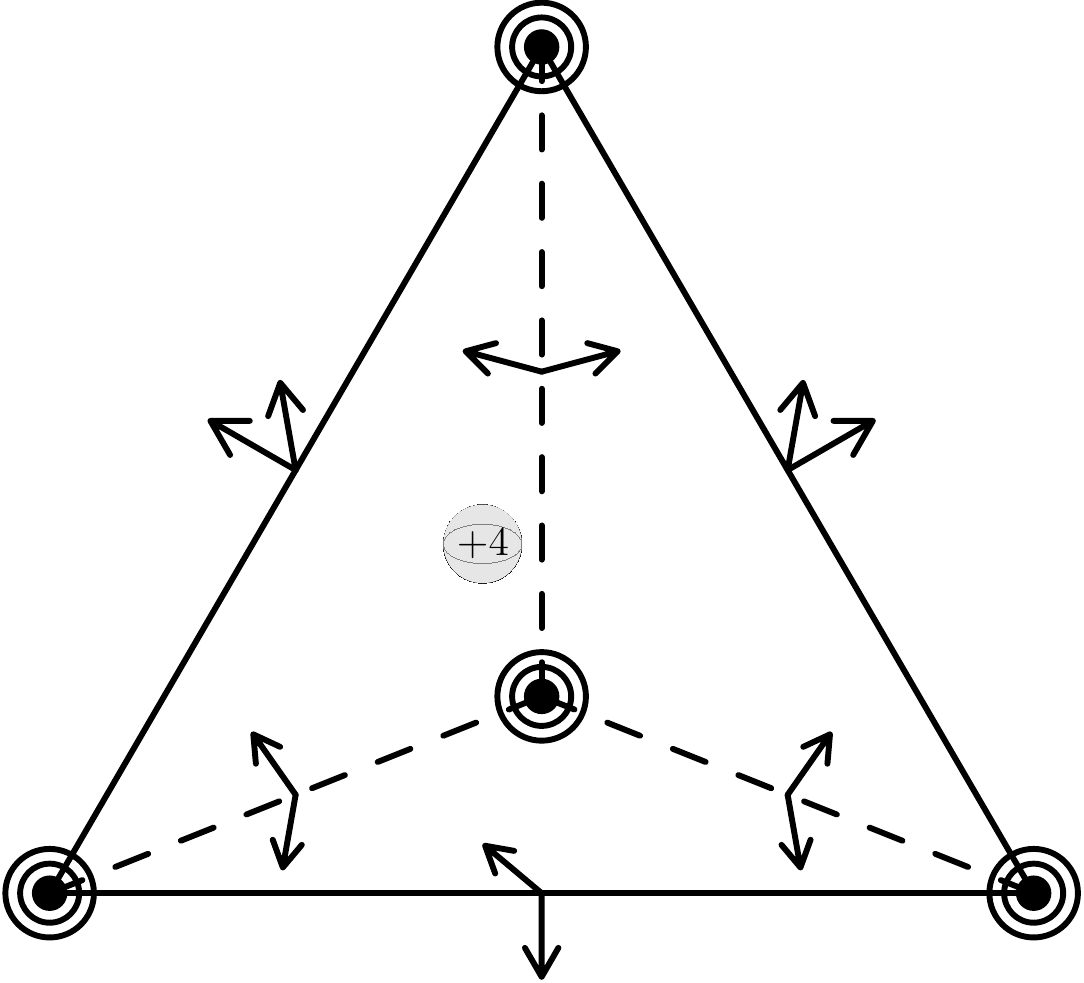}
  \label{fig:Argyris-3Dk5}
}\qquad %
\subfloat[3D Argyris elements, $k=6$]{
  \includegraphics[width=0.28\textwidth]{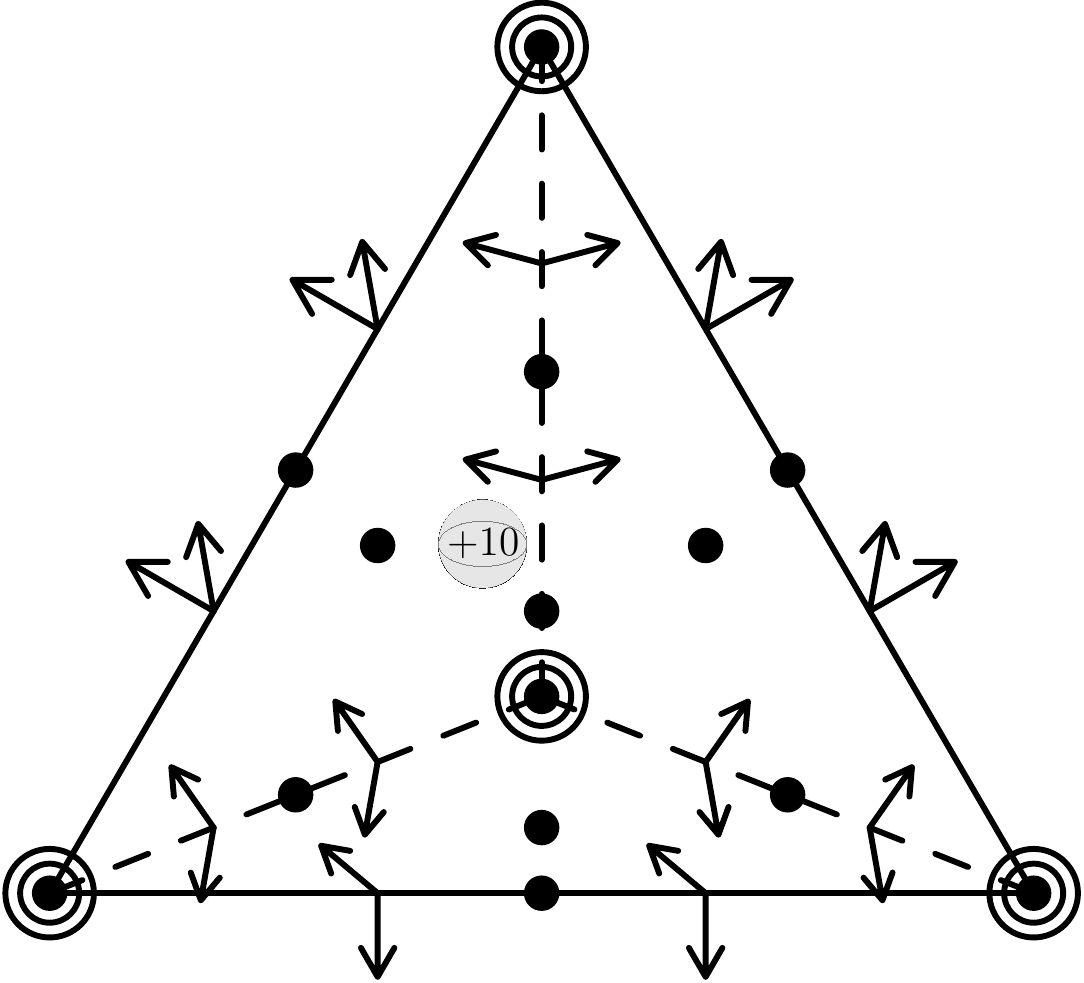}
  \label{fig:Argyris-3Dk6}
} 
\caption{Degrees of freedom of 3D $\mathcal{P}_k$ Argyris elements, in
  the case of $k=5$ and $k=6$}
\label{fig:Argyris-3D}
\vspace{-4mm}
\end{figure}

We sketch the main argument of the unisolvent property. We first
note that the number of degrees of freedom above is 
$$ 
4\times 10 + 6\times(k-5) + 6\times 2(k-4) + 4\times\frac{(k-4)(k-3)}{2}
+ \frac{(k-1)(k-2)(k-3)}{6},
$$ 
which is exactly the dimension of $\mathcal{P}_k(K)$. Therefore,
it suffices to show that $v\in \mathcal{P}_k(K)$ vanishes if it
vanishes at all the degrees of freedom. It is readily seen that the
trace of $v|_F \in \mathcal{P}_k(F)$, which has to be zero since the
degrees of freedom on the face are that of 2D Argyris elements \cite{ciarlet1978finite} (this
also shows the $C^0$-continuity on face). Therefore $v = b_K p$ for
some $p \in \mathcal{P}_{k-4}(K)$, where $b_K$ is the quartic volume
bubble function. By the set of degrees of freedom on element $K$, we
deduce $v \equiv 0$. The $C^1$-continuity on the edge follows from the
$C^1$-continuity of the Argyris elements in 2D. 

\subsection{Finite element spaces}
For every triangulation $\mathcal{T}_h$ of the polytope $\Omega$, we
are now ready to define the finite element spaces $V_h$ as 
\begin{subequations} \label{eq:FEM-space}
\begin{enumerate}
\item For $n = 2$, with $k \geq 3$, 
\begin{equation} \label{eq:FEM-Hermite}
V_h := \{v \in H_0^1(\Omega): v|_K \in \mathcal{P}_k(K), \forall K \in
\mathcal{T}_h, ~v \text{ is }C^1 \text{ at all vertices}\}.
\end{equation}
\item For $n=3$, with $k \geq 5$, 
\begin{equation} \label{eq:FEM-Argyris}
\begin{aligned}
V_h := \{v \in H_0^1(\Omega): v|_K \in \mathcal{P}_k(K), \forall K \in
\mathcal{T}_h, ~& v \text{ is
}C^1 \text{ on all edges}, \\
& v \text{ is }C^2 \text{ at all vertices}\}.
\end{aligned}
\end{equation}
\end{enumerate}
\end{subequations}
A unisolvent set of degrees of freedom of $V_h$ is given locally by
that of 2D Hermite elements or 3D Argyris elements.  Since the finite
elements may have extra continuity on subsimplex, we briefly explain
the implementation of the boundary conditions.  

\paragraph{Boundary condition of 2D Hermite FEM space} Since the first
order derivatives are imposed at vertices, for any $v_h \in V_h$,
the tangential derivatives along the boundary should also be zero due
to the consistency. In this case, we follow the terminology in
\cite{falk2013stokes, christiansen2018nodal} to introduce the
definition of {\it corner vertices}.   

\begin{definition} \label{df:corner-vertex-2D}
A boundary vertex is called a \it{corner vertex} if the two adjacent
boundary edges sharing this vertex do not lie on a straight line.
\end{definition}
      
At each corner boundary vertex, since the two tangential directions
along the boundary form a basis of $\mathbb{R}^2$, we should impose
both the first order derivatives to be zero. At a non-corner boundary
vertex, the two tangential derivatives along its two adjacent edges
coincide up to a sign. In this case, we only specify the degree of freedom value of
this tangential derivative. 

\paragraph{Boundary condition of 3D Argyris FEM space} The enhanced
continuities are imposed on the edge and vertex for the 3D Argris
finite element.  Therefore, we follow the terminology in
\cite{christiansen2018nodal} to introduce the {\it corner vertices and
corner edges}.

\begin{definition}\label{df:corner-3D}
A boundary vertex is called corner vertex in 3D if the adjacent
boundary edges sharing this vertex are not coplanar. A boundary edge
is called a corner edge in 3D if the two adjacent faces on the
boundary sharing this edge are not coplanar.  
\end{definition}

\begin{figure}[!htbp]
\centering 
\captionsetup{justification=centering}
\subfloat[corner vertex $a'$ and corner edge $e'$]{
  \includegraphics[width=0.25\textwidth]{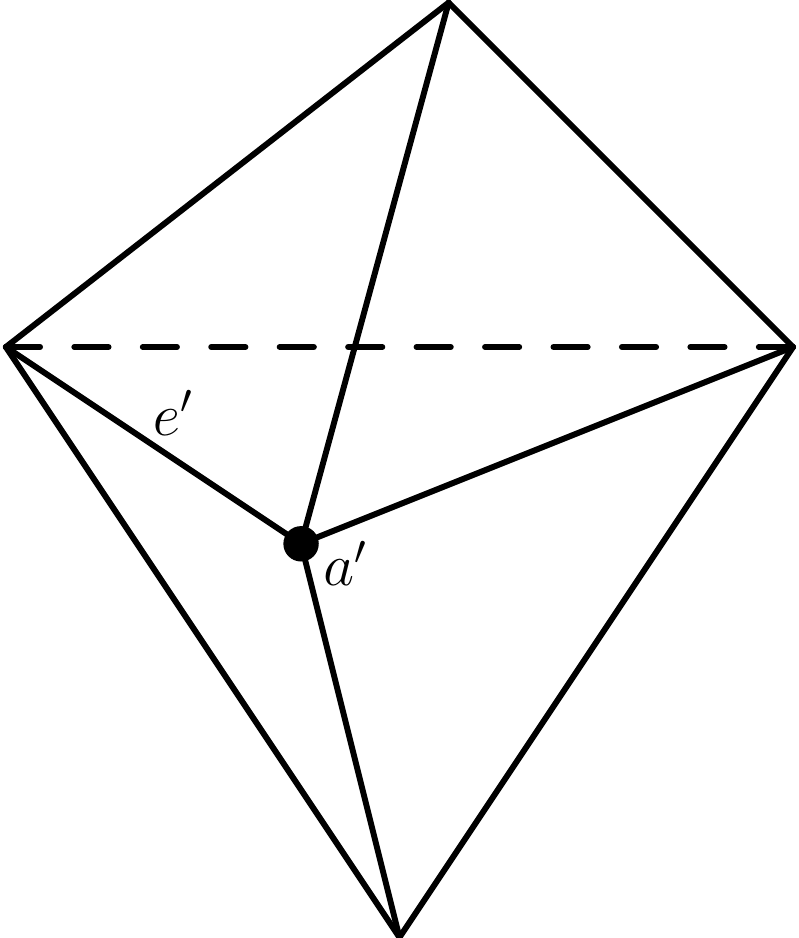}
  \label{fig:3D-corner}
}\qquad %
\subfloat[non-corner vertex $a$ and non-corner edge $e$]{
  \includegraphics[width=0.25\textwidth]{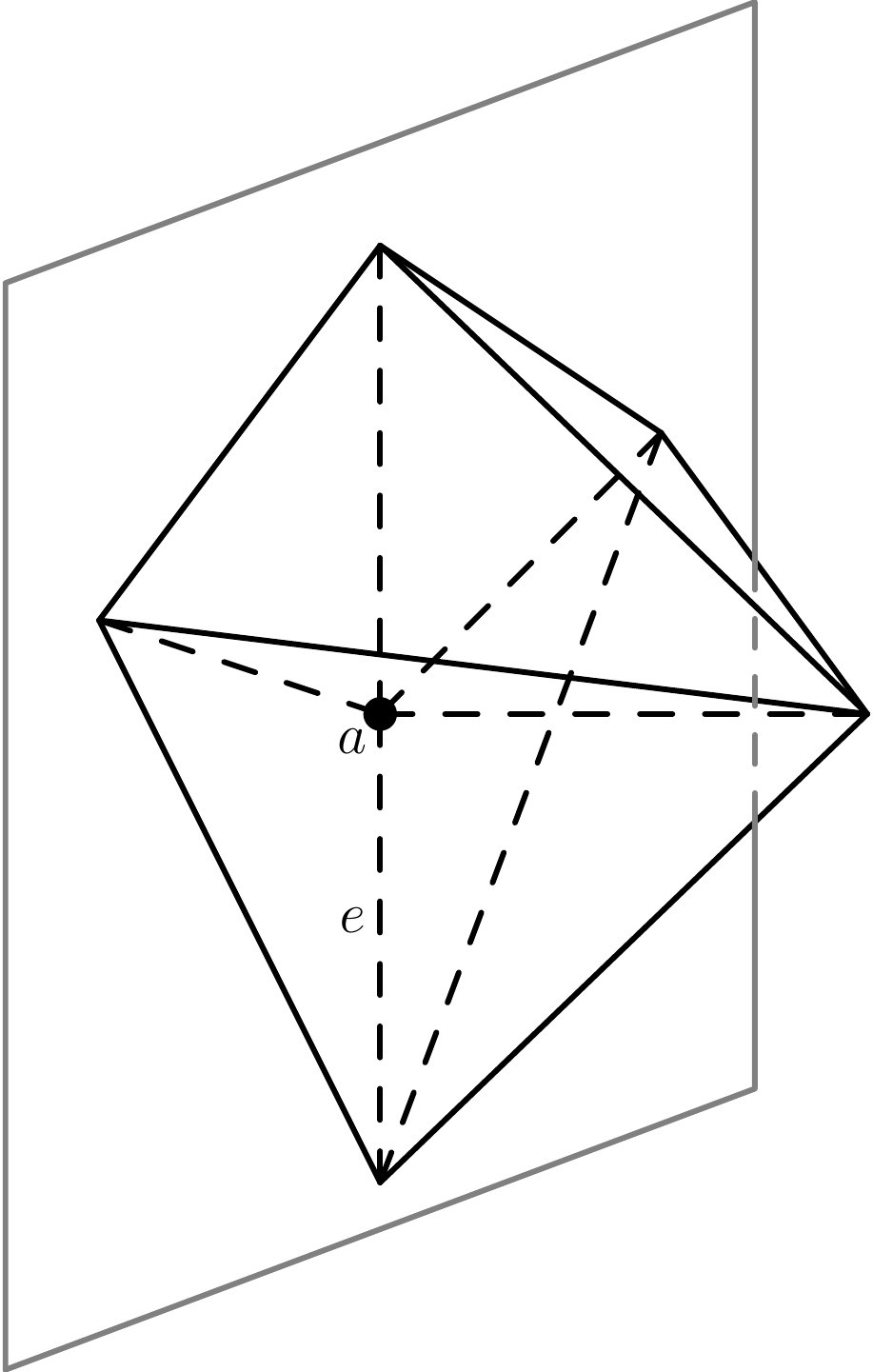}
  \label{fig:3D-non-corner}
} 
\caption{Corner/non-corner vertex and edge in 3D.}
\label{fig:3D-bc}
\vspace{-3mm}
\end{figure}

For a corner boundary vertex in 3D, there are three linearly
independent boundary edges connected to it. Hence, all the degrees of
freedom at the corner boundary vertex should be set to zero.
Similarly, on a corner edge, derivatives of a function along two
normal directions can be determined by the function value on the
boundary. Hence, all the degrees of freedom at the corner boundary
vertex should be set to zero.  

On the other hand, there are only two independent directions, namely
$\bm{t}_1$ and $\bm{t}_2$ along the boundary at a non-corner boundary
vertex or on a non-corner edge (cf. Figure \ref{fig:3D-non-corner}).
Therefore at a non-corner vertex, function value, two tangential
first order derivatives (i.e. $\partial_{\bm{t}_1}$, $\partial_{\bm{t}_2}$)
and three tangential second order derivatives (i.e.
$\partial^2_{\bm{t}_1\bm{t}_1}$, $\partial^2_{\bm{t}_2\bm{t}_2}$,
$\partial^2_{\bm{t}_1\bm{t}_2}$) are set to be zero. The degrees of
freedom corresponding to the normal first order derivative (i.e.
$\partial_{\bm{n}}$) and three second order derivatives (i.e.
$\partial^2_{\bm{n}\bm{n}}$, $\partial^2_{\bm{t}_1\bm{n}}$,
$\partial^2_{\bm{t}_2\bm{n}}$) should be treated as knowns. The treatment 
of degree of freedoms on 
non-corner edge follows a similar way.

\subsection{Discrete Miranda-Talenti-type estimate}

The following lemma is crucial in the design and analysis of $C^0$ (non-Lagrange)
finite element approximations of the linear elliptic equations in non-divergence 
form \eqref{eq:nondiv} and the HJB equations \eqref{eq:HJB}. 

\begin{lemma}[Discrete Miranda-Talenti-type estimate] \label{lm:Hermite-MT}
Let $\Omega \subset \mathbb{R}^n~(n=2,3)$  be a bounded Lipschitz polytopal
domain and $\mathcal{T}_h$ be  a conforming triangulation. For $v \in
V_h$, it holds that  
\begin{equation} \label{eq:Hermite-MT}
\sum_{K\in \mathcal{T}_h} \|\Delta v_h\|_{L^2(K)}^2 = 
\sum_{K\in \mathcal{T}_h} \|D^2 v_h\|_{L^2(K)}^2 + 2\sum_{F \in \mathcal{F}_h^i}
\langle \llbracket \nabla v_h \rrbracket,  \Delta_T v_h \rangle_{F}.
\end{equation} 
\end{lemma}
\begin{proof}
For any simplicial element $K \in \mathcal{T}_h$, the outward unit
normal vector of $\partial K$ is piecewise constant. Using integration by
parts, we obtain (see e.g. \cite[Equ. (3.7)]{smears2013discontinuous}) 
\begin{equation} \label{eq:Hermite-MT1}
\begin{aligned}
\|\Delta v_h\|_{L^2(K)}^2 &=  
\|D^2 v_h\|_{L^2(K)}^2 + \langle \Delta v_h, \frac{\partial
  v_h}{\partial
  \bm{n}} \rangle_{\partial K} - \langle \nabla \frac{\partial
    v_h}{\partial \bm{n}}, \nabla v_h \rangle_{\partial K} \\
 & = \|D^2 v_h\|_{L^2(K)}^2 + \langle \Delta_T v_h, \frac{\partial
  v_h}{\partial
  \bm{n}} \rangle_{\partial K} - \langle \nabla_T \frac{\partial
    v_h}{\partial \bm{n}}, \nabla_T v_h \rangle_{\partial K}.
\end{aligned}
\end{equation}
Here, the common term $\langle \frac{\partial^2 v_h}{\partial \bm{n}^2}, 
\frac{\partial v_h}{\partial \bm{n}} \rangle_{\partial K}$ is cancelled in the last step. We use $\bm{n}_{\bm{t}_F}$ to denote the 
outward unit normal vector of $\partial F$ coplanar to $F$. Using integration by parts on $F \subset \partial K$, we have 
$$ 
\begin{aligned}
\langle \nabla_T \frac{\partial v_h}{\partial \bm{n}}, \nabla_T v_h \rangle_{\partial K} &= \sum_{F\subset \partial K} 
\int_F \nabla_T \frac{\partial v_h}{\partial \bm{n}} \cdot \nabla_T v_h \,\mathrm{d}s \\
&= \sum_{F\subset \partial K} \left(
-\int_F \frac{\partial v_h}{\partial \bm{n}} \Delta_T v_h \,\mathrm{d}s 
+ \int_{\partial F} \frac{\partial v_h}{\partial \bm{n}} \frac{\partial v_h}{\partial \bm{n}_{\bm{t}_F}} \,\mathrm{d}s
\right).
\end{aligned}
$$ 
Hence, \eqref{eq:Hermite-MT1} can be reformulated as 
$$
\|\Delta v_h\|_{L^2(K)}^2 =  
\|D^2 v_h\|_{L^2(K)}^2 + 2 \langle \Delta_T v_h, \frac{\partial
  v_h}{\partial \bm{n}} \rangle_{\partial K} - \sum_{F \subset
    \partial K}
    \langle \frac{\partial v_h}{\partial \bm{n}},
    \frac{\partial v_h}{\partial \bm{n}_{\bm{t}_F}} \rangle_{\partial F}.
$$
Thanks to the $C^0$-continuity, it is readily seen that $\Delta_T v_h$ is continuous across the face
and vanishes on the boundary. Summing over all elements yields
\begin{equation} \label{eq:Hermite-MT2}
\begin{aligned}
\sum_{K\in \mathcal{T}_h}\|\Delta v_h\|_{L^2(K)}^2 &= \sum_{K \in \mathcal{T}_h}\|D^2 v_h\|_{L^2(K)}^2 +
2\sum_{F \in \mathcal{F}_h^i} \langle \llbracket \nabla v_h \rrbracket, \Delta_T v_h \rangle_F \\
& \qquad\qquad - \sum_{K \in \mathcal{T}_h} \sum_{F\subset \partial K} \langle \frac{\partial v_h}{\partial \bm{n}},
    \frac{\partial v_h}{\partial \bm{n}_{\bm{t}_F}} \rangle_{\partial F}.
\end{aligned}
\end{equation}
For any $F\in \mathcal{F}_h^\partial$, the boundary condition implies
that $\frac{\partial v_h}{\partial \bm{n}_{\bm{t}_F}} = 0$.  For any
interior face $F = \partial K^+ \cap \partial K^-$, the
$C^1$-continuity on the $(n-2)$-dimensional subsimplex implies that 
$$ 
\left.\frac{\partial v_h^+}{\partial \bm{n}^+}\right|_{\partial F} = -\left.\frac{\partial v_h^-}{\partial \bm{n}^-}\right|_{\partial F}, \quad 
\left.\frac{\partial v_h^+}{\partial \bm{n}_{\bm{t}_F}}\right|_{\partial F} = \left.\frac{\partial v_h^-}{\partial \bm{n}_{\bm{t}_F}}\right|_{\partial F}.
$$ 
Then, we deduce that the last term in \eqref{eq:Hermite-MT2} vanishes,
which gives the desired result \eqref{eq:Hermite-MT}.
\end{proof}

\section{Applications to the linear elliptic equations in
non-divergence form}
\label{sec:nondiv}
In this section, we apply the $C^0$ (non-Lagrange) finite element method to solve the
linear elliptic equations in non-divergence form \eqref{eq:nondiv}. 

\subsection{Numerical scheme}
Define the broken bilinear form for $w \in V+V_h$ and $v_h \in V_h$:
\begin{equation} \label{eq:nondiv-bilinear}
B_{0,h}(w, v_h) := \sum_{K\in \mathcal{T}_h}(\gamma Lw, \Delta v_h)_K -
(2-\sqrt{1-\varepsilon}) \sum_{F\in \mathcal{F}_h^i} \langle
\llbracket \nabla w\rrbracket, \Delta_T v_h\rangle_F.
\end{equation} 
We propose the following finite element scheme to approximate the
solution to linear elliptic equations in non-divergence form
\eqref{eq:nondiv}: Find $u_h \in V_h$ such that 
\begin{equation} \label{eq:nondiv-h}
B_{0,h}(u_h, v_h) = \sum_{K \in \mathcal{T}_h}(\gamma f, \Delta v_h)_K
  \qquad \forall v_h \in V_h. 
\end{equation}
We emphasize that no penalty or stabilization parameter is involved in
the scheme above. The broken norm is introduced on $V+V_h$:
\begin{equation} \label{eq:nondiv-norm} 
\|v\|_{0,h}^2 := \sum_{K \in \mathcal{T}_h} \|D^2 v\|_{L^2(K)}^2 \quad
\forall v \in V + V_h.
\end{equation} 
Note that $\|v\|_{0,h} = 0$ implies that $D^2v|_K = 0$ on each element
$K$. Together with the $C^1$-continuity on $(n-2)$-dimensional
subsimplex, we immediately have that $v$ is a linear polynomial on
$\Omega$, which means $v\equiv 0$ since $v$ vanishes on
$\partial\Omega$. The following coercivity result follows directly
from the discrete Miranda-Talenti-type estimate in Lemma \ref{lm:Hermite-MT}.

\begin{lemma} \label{lm:nondiv-MT} 
There holds that 
\begin{equation} \label{eq:nondiv-coercivity}
B_{0,h}(v_h, v_h) \geq (1-\sqrt{1-\varepsilon})\|v_h\|_{0,h}^2 \qquad
\forall v_h\in V_h.
\end{equation}
\end{lemma}
\begin{proof}
By using Lemma \ref{lm:nondiv-Cordes-prop} and Cauchy-Schwarz
inequality, we have 
\begin{equation} \label{eq:nondiv-coer2}
\begin{aligned}
B_{0,h}(v_h, v_h) 
&= \sum_{K \in \mathcal{T}_h} (\gamma Lv_h - \Delta v_h,
   \Delta v_h)_K + \sum_{K \in \mathcal{T}_h} \|\Delta v_h\|_{L^2(K)}^2 \\ 
&\qquad\qquad\qquad\qquad 
- (2-\sqrt{1-\varepsilon})\sum_{F\in \mathcal{F}_h^i} \langle
\llbracket\nabla v_h \rrbracket, \Delta_T v_h \rangle_F \\
(\mbox{by } \eqref{eq:nondiv-Cordes-prop})~~ 
& \geq \sum_{K\in \mathcal{T}_h}\|\Delta v_h\|_{L^2(K)}^2 -
\sqrt{1-\varepsilon}\sum_{K\in \mathcal{T}_h} \|D^2v_h\|_{L^2(K)}\|\Delta v_h\|_{L^2(K)}
\\
&\qquad\qquad\qquad\qquad 
-(2-\sqrt{1-\varepsilon}) \sum_{F\in \mathcal{F}_h^i} \langle 
  \llbracket \nabla v_h \rrbracket, \Delta_Tv_h \rangle_F \\
& \geq \sum_{K \in \mathcal{T}_h} \|\Delta v_h\|_{L^2(K)}^2 -
\frac{\sqrt{1-\varepsilon}}{2} \sum_{K\in\mathcal{T}_h} \big(\|D^2v_h\|_{L^2(K)}^2
+ \|\Delta v_h\|_{L^2(K)}^2 \big) \\
&\qquad\qquad\qquad\qquad 
- (2-\sqrt{1-\varepsilon}) \sum_{F\in \mathcal{F}_h^i} \langle
\llbracket \nabla v_h \rrbracket, \Delta_T v_h \rangle_F \\
& = (1-\frac{\sqrt{1-\varepsilon}}{2}) \Big( \sum_{K \in \mathcal{T}_h}\|\Delta
    v_h\|_{L^2(K)}^2 - 2\sum_{F\in \mathcal{F}_h^i}\langle
    \llbracket \nabla v_h \rrbracket, \Delta_T v_h \rangle_F \Big)
\\
&\qquad\qquad\qquad\qquad 
- \frac{\sqrt{1-\varepsilon}}{2} \sum_{K\in \mathcal{T}_h}\|D^2 v_h\|_{L^2(K)}^2 \\
(\mbox{recall } \eqref{eq:nondiv-norm})~~ & = (1-\sqrt{1-\varepsilon})\|v_h\|_{0,h}^2,
\end{aligned}
\end{equation}
where the discrete Miranda-Talenti-type estimate \eqref{eq:Hermite-MT} is
used in the last step.
\end{proof}

We note that the coercivity constant (namely, $1 - \sqrt{1-\varepsilon}$) under the
broken norm $\|\cdot\|_{0,h}$ is exactly the same as that for the PDE theory.  As a corollary, the uniqueness of the finite element method
\eqref{eq:nondiv-h} implies the existence, since equation \eqref{eq:nondiv} is linear.
Further, assume that the solution $u \in
H^2(\Omega)\cap H_0^1(\Omega)$, a straightforward argument shows the
consistency, namely 
\begin{equation} \label{eq:nondiv-consistency} 
B_{0,h}(u,v_h) = \sum_{K\in \mathcal{T}_h} (\gamma L u, \Delta v_h)_K =
\sum_{K \in \mathcal{T}_h} (\gamma f, \Delta v_h)_K 
\qquad \forall v_h \in V_h.
\end{equation}

\begin{remark} \label{rk:tilde-varepsilon}
We note that the bilinear form \eqref{eq:nondiv-bilinear} explicitly uses the constant $\varepsilon$ in Cordes condition \eqref{eq:nondiv-Cordes}. 
In case that the optimal value of $\varepsilon$ is not easy to compute, a simple modification of \eqref{eq:nondiv-bilinear} reads
$$ 
\tilde{B}_{0,h}(w, v_h) := \sum_{K\in \mathcal{T}_h}(\gamma Lw, \Delta v_h)_K -
(2-\sqrt{1-\tilde{\varepsilon}}) \sum_{F\in \mathcal{F}_h^i} \langle
\llbracket \nabla w\rrbracket, \Delta_T v_h\rangle_F,
$$
where $\tilde{\varepsilon}$ is an approximation of $\varepsilon$ that satisfies $\sqrt{1 - \tilde{\varepsilon}} + \frac{1-\varepsilon}{\sqrt{1 - \tilde{\varepsilon}}} < 2$.
Using the inequality 
$$ 
2\|D^2v_h\|_{L^2(K)}\|\Delta v_h\|_{L^2(K)} \leq \frac{\sqrt{1- \varepsilon}}{\sqrt{1- \tilde\varepsilon}}
\|D^2v_h\|_{L^2(K)}^2 + \frac{\sqrt{1- \tilde\varepsilon}}{\sqrt{1- \varepsilon}} \|\Delta v_h\|_{L^2(K)}^2 \quad \forall K \in \mathcal{T}_h,
$$ 
we have the coercivity result 
$$
\tilde{B}_{0,h}(v_h, v_h) \geq \big( 1 - \frac{\sqrt{1 - \tilde\varepsilon}}{2} - \frac{1-\varepsilon}{2\sqrt{1 - \tilde\varepsilon}} \big) \|v_h\|_{0,h}^2 \qquad \forall v_h \in V_h,
$$
following a similar argument as Lemma \ref{lm:nondiv-MT}. Clearly, the optimal coercivity constant is attained at $\tilde{\varepsilon} = \varepsilon$. Even if there is no any a priori estimate of $\varepsilon$, one may simply take $\tilde{\varepsilon} = 0$, which leads to the coercivity constant $\frac{\varepsilon}{2}$.
\end{remark}

\subsection{Error estimate}
Thanks to the coercivity result \eqref{eq:nondiv-coercivity} and the
consistency \eqref{eq:nondiv-consistency}, we then arrive at the
quasi-optimal error estimate. 
\begin{theorem} \label{tm:nondiv-estimate}
Let $\Omega$ be a bounded, convex polytope in $\mathbb{R}^n$, and let
$\mathcal{T}_h$ be a simplicial, conforming, shape-regular mesh.
Suppose that the coefficients satisfy the Cordes condition
\eqref{eq:nondiv-Cordes} and the solution to \eqref{eq:nondiv} $u \in
H^s(\Omega) \cap
H_0^1(\Omega)$ for some $s \geq 2$. Then, there holds 
\begin{equation} \label{eq:nondiv-estimate}
\|u - u_h\|_{0,h}^2:= \sum_{K \in \mathcal{T}_h} \|D^2(u - u_h)\|_{K}^2 \leq C
\sum_{K\in \mathcal{T}_h} h_K^{2t-4} \|u\|_{H^s(K)}^2,
\end{equation}
where $t = \min\{s, k+1\}$.
\end{theorem}
\begin{proof}
Since the sequence of meshes is shape regular, it follows from the
standard polynomial approximation theory \cite{brenner2007mathematical} that, there exists a $z_h \in
V_h$, such that
\begin{subequations}
\begin{align} 
\|u - z_h\|_{H^q(K)} &\leq C h_K^{t-q}\|u\|_{H^s(\omega_K)}, \quad 0 \leq q \leq 2,
  \label{eq:nondiv-app1}\\
\|D^\beta(u - z_h)\|_{L^2(\partial K)} &\leq C h_K^{t-q-1/2}
\|u\|_{H^s(\omega_K)} \quad \forall |\beta| = q, \quad 0 \leq q \leq 1, \label{eq:nondiv-app2}
\end{align}
\end{subequations}
where $\omega_K$ represents the union of the local neighborhood of element $K$. 
Let $\psi_h = z_h - u_h$. Then, by the coercivity result
\eqref{eq:nondiv-coercivity}, we obtain  
\begin{equation} \label{eq:nondiv-error}
\begin{aligned}
\|z_h - u_h\|_{0,h}^2 &\lesssim B_{0,h}(z_h - u_h, \psi_h) =
B_{0,h}(z_h, \psi_h) - \sum_{K\in \mathcal{T}_h}(\gamma f, \Delta \psi_h)_K  \\
&= \underbrace{\sum_{K \in \mathcal{T}_h}(\gamma L  (z_h - u), \Delta
\psi_h)_K}_{E_1} 
- 
(2 - \sqrt{1-\varepsilon}) \underbrace{\sum_{F\in \mathcal{F}_h^i} \langle 
\llbracket \nabla z_h \rrbracket, \Delta_T \psi_h \rangle_F}_{E_2}.
\end{aligned}
\end{equation}

By the boundedness of the data, the fact that
$\|\Delta\psi_h\|_{L^2(K)} \leq \sqrt{n} \|D^2
\psi_h\|_{L^2(K)}$ and the approximation result
\eqref{eq:nondiv-app1}, we have 
$$ 
\begin{aligned}
|E_1| &\leq  
\sum_{K \in \mathcal{T}_h} 
\|\gamma L (u-z_h)\|_{L^2(K)}  \|\Delta \psi_h\|_{L^2(K)} \\ 
&\leq \sum_{K \in \mathcal{T}_h} 
\sqrt{n} \|\gamma\|_{L^\infty(K)} \|A\|_{L^\infty(K)}\|D^2(u -
    z_h)\|_{L^2(K)}\|D^2\psi_h\|_{L^2(K)} \\
& \lesssim \Big( 
  \sum_{K \in \mathcal{T}_h} h_K^{2t-4}\|u\|_{H^s(K)}^2 
  \Big)^{1/2} \|\psi_h\|_{0,h}.
\end{aligned}
$$ 
Further, the local trace inequality implies that $\|\Delta_T \psi_h
\|_{L^2(\partial K)} \lesssim h_K^{-1/2}\|D^2 \psi_h\|_{L^2(K)}$,
together with the approximation result \eqref{eq:nondiv-app2}, we have  
$$ 
\begin{aligned}
|E_2| & =   
\left| \sum_{F \in \mathcal{F}_h} 
\langle\llbracket \nabla u - \nabla z_h \rrbracket, \Delta_T \psi_h
\rangle_F \right| \leq \sum_{F \in \mathcal{F}_h^i} \|\llbracket \nabla
(u-z_h)\rrbracket \|_{L^2(F)} \|\Delta_T \psi_h\|_{L^2(F)} \\
& \lesssim \Big( 
  \sum_{K \in \mathcal{T}_h} h_K^{2t-4}\|u\|_{H^s(K)}^2 
  \Big)^{1/2} \|\psi_h\|_{0,h}.
\end{aligned}
$$ 
The above inequalities and \eqref{eq:nondiv-error} give rise to 
$$ 
\|z_h - u_h\|_{0,h} \leq C \Big( \sum_{K\in \mathcal{T}_h} h_K^{2t-4}
\|u\|_{H^s(K)}^2 \Big)^{1/2},
$$ 
which implies the desired result by triangle inequality. 
\end{proof}

\section{Applications to the Hamilton-Jacobi-Bellman equations} 
\label{sec:HJB}
In this section, we apply the $C^0$ (non-Lagrange) finite element method to solve the
HJB equations \eqref{eq:HJB}, which can be viewed as a natural extension of the numerical scheme for the linear elliptic equations in non-divergence form. 
Since an additional parameter $\lambda >
0$ is introduced in the Cordes condition \eqref{eq:HJB-Cordes1}, the
broken norm is defined as 
\begin{equation} \label{eq:HJB-norm}
\|v\|_{\lambda, h}^2 := \sum_{K\in \mathcal{T}_h} \|v\|_{\lambda, h,
K}^2 := \sum_{K \in \mathcal{T}_h} \big(\|D^2 v\|_{L^2(K)}^2 +
2\lambda \|\nabla v\|_{L^2(K)}^2 + \lambda^2 \|v\|_{L^2(K)}^2\big).
\end{equation}
Thanks to the discussion of $\|\cdot\|_{0,h}$ in
\eqref{eq:nondiv-norm}, it is readily seen that $\|\cdot\|_{\lambda, h}$
is indeed a norm on $V + V_h$ for all $\lambda \geq 0$.

\subsection{Numerical scheme}
We describe the finite element method. In light of
\eqref{eq:HJB-M}, we define the operator $M_h: V+V_h \to V_h^*$
by 
\begin{equation} \label{eq:HJB-form}
\langle M_h[w], v_h \rangle := \sum_{K \in \mathcal{T}_h} (F_\gamma[w],
L_\lambda v_h)_K - (2-\sqrt{1-\varepsilon})\sum_{F\in \mathcal{F}_h^i}
\langle \llbracket \nabla w \rrbracket, \Delta_T v_h - \lambda
v_h\rangle_F,
\end{equation} 
where we recall that $L_\lambda v = \Delta v - \lambda v$ in
\eqref{eq:L-lambda}.  The following finite element method is proposed
to approximate the solution to the HJB equations  \eqref{eq:HJB}: Find
$u_h \in V_h$ such that 
\begin{equation} \label{eq:HJB-h}
\langle M_h[u_h], v_h \rangle = 0 \qquad \forall v_h \in V_h. 
\end{equation}

\begin{lemma} \label{lm:HJB-monotone}
For every $w_h, v_h \in V_h$, we have 
\begin{equation} \label{eq:HJB-monotone} 
\langle M_h[w_h] - M_h[v_h], w_h-v_h \rangle 
\geq (1-\sqrt{1-\varepsilon})\|w_h - v_h\|_{\lambda, h}^2.
\end{equation} 
\end{lemma}
\begin{proof}
Set $z_h = w_h - v_h$. Using the discrete Miranda-Talenti-type
estimate \eqref{eq:Hermite-MT} and integration by parts, we obtain 
\begin{equation} \label{eq:Hermite-MT-lambda}
\begin{aligned}
& \quad~ \sum_{K \in \mathcal{T}_h} \|L_\lambda z_h\|_{L^2(K)}^2 \\
&= 
\sum_{K\in \mathcal{T}_h}\|\Delta z_h\|_{L^2(K)}^2 - 2\lambda\sum_{K\in
\mathcal{T}_h} (z_h, \Delta z_h)_K + \lambda^2 \|z_h\|_{L^2(\Omega)}^2
\\
& = 
\sum_{K\in \mathcal{T}_h}\|\Delta z_h\|_{L^2(K)}^2 + 2\lambda \|\nabla
z_h\|_{L^2}^2 + \lambda^2 \|z_h\|_{L^2(\Omega)}^2 - 2\lambda\sum_{K\in
\mathcal{T}_h} \int_{\partial K} z_h \frac{\partial z_h}{\partial
\bm{n}} \,\mathrm{d}s \\
 & = \|z_h\|_{\lambda, h}^2 + 2 \sum_{F\in \mathcal{F}_h^i} \langle
\llbracket \nabla z_h \rrbracket, \Delta_Tz_h - \lambda z_h \rangle_F,
\end{aligned}
\end{equation}
where we use the definition of $\|\cdot\|_{\lambda, h}$
\eqref{eq:HJB-norm} in the last step. Further, by Lemma
\ref{lm:HJB-Cordes-prop}, we have 
$$ 
\begin{aligned}
& \quad~ \langle M_h[w_h] - M_h[v_h], z_h \rangle \\
&= \sum_{K\in \mathcal{T}_h} (F_\gamma[w_h] - F_\gamma[v_h] -
L_\lambda z_h, L_\lambda z_h)_K + \sum_{K \in \mathcal{T}_h} \|L_\lambda
z_h\|_{L^2(K)}^2 \\ 
& \qquad\qquad\qquad
- (2-\sqrt{1-\varepsilon})\sum_{F\in \mathcal{F}_h^i} \langle
\llbracket \nabla z_h \rrbracket, \Delta_T z_h - \lambda z_h \rangle_F \\
&\geq \sum_{K\in \mathcal{T}_h}\|L_\lambda z_h\|_{L^2(K)}^2 -
\sqrt{1-\varepsilon}  \sum_{K\in \mathcal{T}_h} \|z_h\|_{\lambda, h,K} \|L_\lambda
z_h\|_{L^2(K)}
\\
& \qquad\qquad\qquad 
- (2-\sqrt{1-\varepsilon})\sum_{F\in \mathcal{F}_h^i} \langle
\llbracket \nabla z_h \rrbracket, \Delta_T z_h -\lambda z_h\rangle_F
\\
&\geq \frac{2-\sqrt{1-\varepsilon}}{2}\big( 
\sum_{K\in \mathcal{T}_h} \|L_\lambda z_h\|_{L^2(K)}^2 - 2\sum_{F\in \mathcal{F}_h^i}
\langle \llbracket \nabla z_h \rrbracket, \Delta_T z_h -\lambda
z_h\rangle_F \big) \\
&\qquad\qquad\qquad 
- \frac{\sqrt{1-\varepsilon}}{2} \|z_h\|_{\lambda, h}^2.
\end{aligned}
$$ 
Applying \eqref{eq:Hermite-MT-lambda}, we conclude the strong
monotonicity of $M_h$ in \eqref{eq:HJB-monotone}. 
\end{proof}

Again, the monotonicity constant under the broken norm is exactly the
same as that for the PDE theory. Similar to the Remark \ref{rk:tilde-varepsilon}, the monotonicity constant becomes $1 - \frac{\sqrt{1 - \tilde\varepsilon}}{2} - \frac{1-\varepsilon}{2\sqrt{1 - \tilde\varepsilon}}$ if $\varepsilon$ is replaced by its apporixomation $\tilde\varepsilon$.
Next, we show that $M_h$ is Lipschitz
continuous on $V_h$ with respect
to $\|\cdot\|_{\lambda, h}$. 
\begin{lemma} \label{lm:HJB-Lipschitz}
For any $v_h, w_h, z_h \in V_h$,
\begin{equation} \label{eq:HJB-Lipschitz}
|\langle M_h[w_h] - M_h[v_h], z_h\rangle| \leq C
\|w_h-v_h\|_{\lambda,h} \|z_h\|_{\lambda,h}.
\end{equation}
\end{lemma}
\begin{proof}
In light of the definition of $M_h$ in \eqref{eq:HJB-form}, using
Cauchy-Schwarz inequality, we have  
$$ 
\begin{aligned}
& \quad |\langle M_h[w_h] - M_h[v_h], z_h \rangle| \\
& \leq \underbrace{\sum_{K \in \mathcal{T}_h}\|F_\gamma[w_h] -
F_\gamma[v_h] - L_\lambda(w_h - v_h)\|_{L^2(K)} \|L_\lambda
  z_h\|_{L^2(K)}}_{I_1}
\\ 
& \quad 
+ \underbrace{\sum_{K\in \mathcal{T}_h} \|L_\lambda(w_h - v_h)\|_{L^2(K)}
  \|L_\lambda z_h\|_{L^2(K)}}_{I_2} \\
& \quad +
(2-\sqrt{1-\varepsilon}) \underbrace{\sum_{F\in
  \mathcal{F}_h^i}\|\llbracket \nabla(w_h-v_h) \rrbracket \|_{L^2(F)}
\Big( 
\|\Delta_T z_h\|_{L^2(F)} + \|\lambda z_h\|_{L^2(F)} 
\Big)}_{I_3}. 
\end{aligned}
$$ 
Revoking Lemma \ref{lm:HJB-Cordes-prop}, and the fact that
$\sum_{K\in \mathcal{T}_h} \|L_\lambda v\|_{L^2(\mathcal{T}_h)}^2 \leq 2n \|v\|_{\lambda,
h}^2$ for any $v\in V+V_h$, we have 
$$ 
I_1 \leq \sqrt{2n(1-\varepsilon)}\|w_h -
v_h\|_{\lambda,h}\|z_h\|_{\lambda,h}, \qquad 
I_2 \leq 2n\|w_h - v_h\|_{\lambda, h} \|z_h\|_{\lambda, h}.
$$ 
For any interior face $F = \partial K^+ \cap \partial K^-$, the
standard scaling argument \cite{ciarlet1978finite, brenner2007mathematical} gives 
$$ 
\|\llbracket \nabla(w_h - v_h) \rrbracket \|_{L^2(F)}^2 \lesssim
h_F \sum_{K \in \{K^+, K^-\}}\|D^2 (w_h - v_h)\|_{L^2(K)}^2,
$$ 
where the $C^0$-continuity at face and $C^1$-continuity at
$(n-2)$-dimensional subsimplex guarantee that the piecewise linear
function on $\omega_F  = K^+ \cup K^-$ has to be a linear function on the $\omega_F$. 
Further, By the local trace inequality, we have that for $F
\subset \partial K$ 
$$ 
\|\Delta_T z_h\|_{L^2(F)}^2 \lesssim h_F^{-1} \|D^2 z_h\|_{L^2(K)}, 
\qquad \|\lambda z_h\|_{L^2(F)}^2 \lesssim h_F^{-1}
\lambda^2 \|z_h\|_{L^2(K)}^2.
$$ 
Hence, we have $I_3 \lesssim \|w_h - v_h\|_{\lambda,h}
\|z_h\|_{\lambda,h}$. The bound \eqref{eq:HJB-Lipschitz} is obtained
from the above estimates of $I_i~(i=1,2,3)$.
\end{proof}
Having the strong monotonicity and the Lipschitz continuity, by the
Browder-Minty Theorem, there exists a unique solution $u_h \in V_h$ to
\eqref{eq:HJB-h}.  

\subsection{Error estimate}
The consistency of \eqref{eq:HJB-h} follows naturally since the term
$\sum_{F\in \mathcal{F}_h^i} \langle \llbracket \nabla u \rrbracket,
\Delta_T v_h - \lambda v_h\rangle_F$ vanishes for $u \in H^2(\Omega)
\cap H_0^1(\Omega)$.  Finally, we arrive at the quasi-optimal error
estimate. 

\begin{theorem} \label{tm:HJB-estimate}
Let $\Omega$ be a bounded, convex polytope in $\mathbb{R}^n$, and let
$\mathcal{T}_h$ be a simplicial, conforming, shape-regular mesh. Let
$\Lambda$ be a compact metric space.  Suppose that the coefficients
satisfy the Cordes condition \eqref{eq:HJB-Cordes}. Then, there exists
a unique solution $u_h \in V_h$ satisfying \eqref{eq:HJB-h}.
Moreover, there holds that 
\begin{equation} \label{eq:HJB-estimate}
\|u - u_h\|_{\lambda,h}^2 \leq C
\sum_{K\in \mathcal{T}_h} h_K^{2t-4} \|u\|_{H^s(K)}^2,
\end{equation}
where $t = \min\{s, k+1\}$ provided that $u \in H^s(\Omega) \cap
H_0^1(\Omega)$ for some $s \geq 2$.
\end{theorem}
\begin{proof}
Since the sequence of meshes is shape regular, it follows from the
standard polynomial approximation theory \cite{brenner2007mathematical} that, there exists a $z_h \in
V_h$, such that
\begin{subequations}
\begin{align} 
\|u - z_h\|_{H^q(K)} &\leq C h_K^{t-q}\|u\|_{H^s(\omega_K)}, \quad 0 \leq q \leq 2,
  \label{eq:HJB-app1}\\
\|D^\beta(u - z_h)\|_{L^2(\partial K)} &\leq C h_K^{t-q-1/2}
\|u\|_{H^s(\omega_K)} \quad \forall |\beta| = q, \quad 0 \leq q \leq 1. \label{eq:HJB-app2}
\end{align}
\end{subequations}
Let $\psi_h = z_h - u_h$. In light of the consistency, the strong
monotonicity of $M_h$ on $V_h$, as shown in Lemma
\ref{lm:HJB-monotone}, yields 
\begin{equation} \label{eq:HJB-error}
\begin{aligned}
\|\psi_h\|_{\lambda, h}^2 &\lesssim \langle M_h[z_h] - M_h[u_h], \psi_h
\rangle = \langle M_h[z_h] - M[u], \psi_h \rangle \\
&= \underbrace{\sum_{K \in \mathcal{T}_h} (F_\gamma[z_h] - F_\gamma[u] -
    L_\lambda(z_h - u), L_\lambda \psi_h)_K}_{E_1} \\ 
&\quad + \underbrace{\sum_{K \in \mathcal{T}_h}(L_\lambda(z_h - u),
    L_\lambda \psi_h)_K}_{E_2} \\
&\quad - (2-\sqrt{1-\varepsilon}) \underbrace{\sum_{F\in
\mathcal{F}_h^i} \langle \llbracket \nabla z_h \rrbracket, \Delta_T
  \psi_h - \lambda \psi_h \rangle_F}_{E_3}.
\end{aligned}
\end{equation} 
Similar to the proof of Lemma \ref{lm:HJB-Lipschitz}, we obtain 
$$ 
\begin{aligned}
|E_1| &\leq \sqrt{2n(1-\varepsilon)} \|u - z_h\|_{\lambda, h}
\|\psi_h\|_{\lambda, h} 
\lesssim 
\Big( \sum_{K \in \mathcal{T}_h} h_K^{2t-4}\|u\|_{H^s(K)}^2
\Big)^{1/2} \|\psi_h\|_{\lambda,h}, \\
|E_2| &\leq 2n \|u - z_h\|_{\lambda, h} \|\psi_h\|_{\lambda,h}
\lesssim 
\Big( \sum_{K \in \mathcal{T}_h} h_K^{2t-4}\|u\|_{H^s(K)}^2
\Big)^{1/2} \|\psi_h\|_{\lambda,h}.
\end{aligned}
$$ 
By \eqref{eq:HJB-app2} and the local trace inequality, we have   
$$ 
\begin{aligned}
|E_3| 
& = \Big|\sum_{F\in \mathcal{F}_h^i} \langle \llbracket \nabla
(u-z_h) \rrbracket, \Delta_T \psi_h - \lambda \psi_h\rangle_F \Big| \\
& \lesssim
\Big|\sum_{F\in \mathcal{F}_h^i} \|\llbracket \nabla
(u-z_h) \rrbracket\|_{L^2(F)} 
\big( 
\|\Delta_T \psi_h\|_{L^2(F)} + \lambda\|\psi_h\|_{L^2(F)} \big) \Big|
\\
& \lesssim 
\Big( \sum_{K \in \mathcal{T}_h} h_K^{2t-4}\|u\|_{H^s(K)}^2
\Big)^{1/2} \|\psi_h\|_{\lambda,h}.
\end{aligned}
$$ 
The above estimates of $E_i~(i=1,2,3) $ and \eqref{eq:HJB-error} yield 
$$ 
\|z_h - u_h\|_{\lambda,h} \leq C \Big( \sum_{K\in \mathcal{T}_h} h_K^{2t-4}
\|u\|_{H^s(K)}^2 \Big)^{1/2},
$$ 
which implies the desired result by triangle inequality. 
\end{proof}

\subsection{Semismooth Newton method}
We use the semismooth Newton method \cite{ulbrich2002semismooth} to
solve the discrete problem \eqref{eq:HJB-h}. We follow a similar
argument as \cite{smears2014discontinuous} in this subsection.  Since
transferring the proofs in \cite{smears2014discontinuous} to our
setting is straightforward, we only describe the algorithm and the
convergence result. 

Following the discussion in \cite{smears2014discontinuous}, we define
the set of admissible maximizers for any $v \in V + V_h$, 
\begin{equation} \label{eq:maximizer}
\Lambda[v] := 
\left\{ 
\parbox{5.2em}{
$g: \Omega \to \Lambda$ \\ 
measurable} 
\Bigg|~ 
\parbox{20em}{ 
$
\displaystyle g(x) \in
\mathop{\arg\max}_{\alpha\in\Lambda}(A^\alpha:D_h^2v + \bm{b}^\alpha
\cdot \nabla v - c^\alpha v - f^\alpha)$ \\
for almost every $x\in \Omega$
}
\right\},
\end{equation}
where $D_h^2 v$ denotes the broken Hessian of $v$.  As shown in
\cite[Lemma 9, Theorem 10]{smears2014discontinuous}, the set
$\Lambda[v]$ is nonempty for any $v \in V + V_h$, where a selection
theorem in \cite{kuratowski1965general} is applied. For any measurable
$g(x): \Omega \to \Lambda$, thanks to the uniform continuity of
$\gamma^\alpha$ defined in \eqref{eq:HJB-gamma} on $\Omega \times
\Lambda$, $\gamma^g := \gamma^\alpha|_{\alpha = g(x)}$ satisfies
$\gamma^{g}\in L^\infty(\Omega)$ and $\|\gamma^g\|_{L^\infty(\Omega)}
\leq \|\gamma^\alpha\|_{C(\bar{\Omega} \times \Lambda)}$. The
functions $A^g$, $\bm{b}^g$, $c^g$ and $f^g$ and the operator $L^g$
are defined in a similar way and are likewise bounded.  

The semismooth Newton algorithm for solving \eqref{eq:HJB-h} is
described as follows. 

\vspace{2mm}
\noindent {\bf Input:} Given initial guess $u_h^0 \in V_h$ and a
stopping criterion. 

\noindent {\bf for} $j = 0,1,2,\cdots$ {\bf until} termination {\bf
do}

Choose any $\alpha_j \in \Lambda[u_h^j]$ and compute compute
$u_h^{j+1} \in V_h$ as the solution to the linear problem 
\begin{equation} \label{eq:semismooth-algo}
B^{j}_{\lambda, h}(u_h^{j+1}, v_h) = \sum_{K \in \mathcal{T}_h} (
\gamma^{\alpha_j} f^{\alpha_j}, L_\lambda v_h)_K \qquad \forall v_h
\in V_h,   
\end{equation}
where the bilinear form $B^j_{\lambda, h}: V_h \times V_h \to
\mathbb{R}$ is defined by 
\begin{equation} \label{eq:semismooth-linear}
\begin{aligned}
B^j_{\lambda, h}(w_h, v_h) &:= \sum_{K\in \mathcal{T}_h}
(\gamma^{\alpha_j} L^{\alpha_j}w_h, L_\lambda v_h)_K \\
&\qquad 
- (2-\sqrt{1-\varepsilon})\sum_{F \in \mathcal{F}_h^i} \langle
\llbracket \nabla w_h \rrbracket, \Delta_T v_h - \lambda v_h
\rangle_{F}.
\end{aligned}
\vspace{-4mm}
\end{equation}
\noindent {\bf end do}

\begin{remark} \label{rk:lower-order}
We note here that \eqref{eq:semismooth-algo} is indeed a finite
element scheme for solving the linear elliptic equations in
non-divergence form with lower-order terms: 
$$ 
L^{\alpha_j} u^{j+1} := A^{\alpha_j}:D^2u^{j+1} +
\bm{b}^{\alpha_j}\cdot \nabla u^{j+1}
- c^{\alpha_j} u^{j+1} = f^{\alpha_j}
\quad \text{in }\Omega, \quad 
u^{j+1} = 0 \quad\text{on }\partial\Omega,
$$ 
where the coefficients, which is allowed to be discontinuous, satisfy
a similar Cordes condition as \eqref{eq:HJB-Cordes} with $\alpha =
\alpha_j$. Using very similar arguments as Lemma \ref{lm:HJB-monotone}
and Lemma \ref{lm:HJB-Lipschitz}, the coercivity and boundedness of
$B_{\lambda, h}^j$ regarding to $\|\cdot\|_{\lambda, h}$ can be
proved. A quasi-optimal error estimate then follows directly, which is
also confirmed numerically in subsection \ref{subsec:experiment2}.
\end{remark}

We state the main result as follows. 
\begin{theorem} \label{tm:semismooth-convergence}
Under the hypotheses of Theorem \ref{tm:HJB-estimate}, there exists a
constant $R > 0$ that may depend on $h$ as well as on the polynomial
degree, such that if $\|u_h - u_h^0\|_{\lambda, h} < R$, then the
sequence $\{u_h^j\}_{j=1}^\infty$ generated by the semismooth Newton
algorithm converges to $u_h$ with a superlinear convergence rate. 
\end{theorem}
\begin{proof}
The proof is similar to \cite[Theorem 11]{smears2014discontinuous} and
is therefore omitted here. 
\end{proof}

\section{Numerical experiments} \label{sec:numerical}
In this section we present some numerical experiments of the $C^0$
(non-Lagrange) finite element methods for the linear elliptic
equations in non-divergence form \eqref{eq:nondiv} and the HJB
equations \eqref{eq:HJB}. For all the convergence order experiments,
the convergence history plots are logarithmically scaled.  

%%%%%%%%%%%%%%%%%%%%%%%%%%%%%%%%%%%%%%%%%%%%%%%%%%%%%%%%%%%%
%% Experiment 1
%%%%%%%%%%%%%%%%%%%%%%%%%%%%%%%%%%%%%%%%%%%%%%%%%%%%%%%%%%%%
\subsection{First experiment}
In the first experiment, we consider the problem
\eqref{eq:nondiv} in two dimensions on the domain $\Omega = (-1,1)^2$.
The coefficient matrix is set to be 
\begin{equation} \label{eq:nondiv-test1} 
A = \begin{pmatrix} 
2 & \frac{x_1 x_2}{|x_1 x_2|} \\
\frac{x_1 x_2}{|x_1 x_2|} & 2 
\end{pmatrix}.
\end{equation} 
A straightforward calculation shows that, for the coefficient matrix
in \eqref{eq:nondiv-test1}, the Cordes condition
\eqref{eq:nondiv-Cordes} is satisfied with $\varepsilon = 3/5$. We
note here that the coefficient matrix is discontinuous across the set
$\{(x_1,x_2)\in \Omega: x_1=0 \text{ or } x_2=0\}$.  In order to test
the convergence order, the smooth solution 
\begin{equation} \label{eq:nondiv-test1-u}
u(x) = (x_1 \mathrm{e}^{1-|x_1|} - x_1)(x_2 \mathrm{e}^{1-|x_2|} - x_2)
\end{equation}
is considered from many works (e.g., \cite{smears2013discontinuous,
gallistl2017variational}).  The right hand side $f:= A:D^2u$ is
directly calculated from the coefficient matrix and solution. 

On a sequence of uniform triangulations $\{\mathcal{T}_h\}_{0<h<1}$,
we apply the numerical scheme \eqref{eq:nondiv-h} to the problem with
2D Hermite finite element spaces for polynomial degrees $k=3$ and
$k=4$.  After computing \eqref{eq:nondiv-h} for various $h$, we report
the errors in Figure \ref{fig:test1}. The expected optimal convergence
rate $\|D^2 u -D^2 u_h\|_{L^2(\mathcal{T}_h)} = \mathcal{O}(h^{k-1})$
is observed, which is in agreement with Theorem
\ref{tm:nondiv-estimate}.  Further, the experiments indicate that the
scheme converges with (sub-optimal) second order convergence in both
$H^1$ and $L^2$ when $k=3$. As for $k=4$, the $H^1$ error converges
with (optimal) fourth order, and the $L^2$ error converges with
(sub-optimal) fourth order. 

\begin{figure}[!htbp]
\centering 
\captionsetup{justification=centering}
\subfloat[Convergence rate, $k=3$]{
  \includegraphics[width=0.48\textwidth]{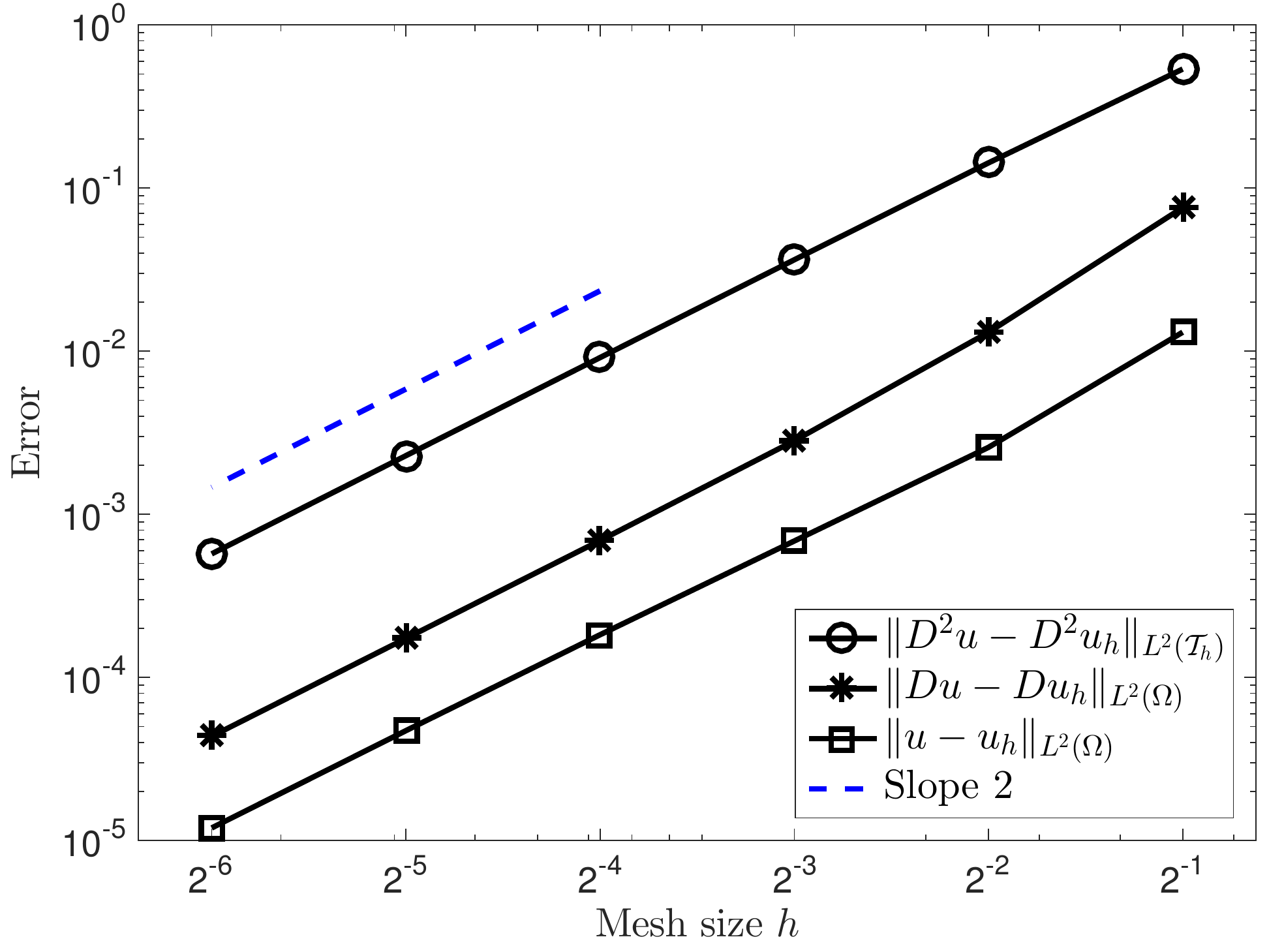}
  \label{fig:test1-k3}
}%
\subfloat[Convergence rate, $k=4$]{
  \includegraphics[width=0.48\textwidth]{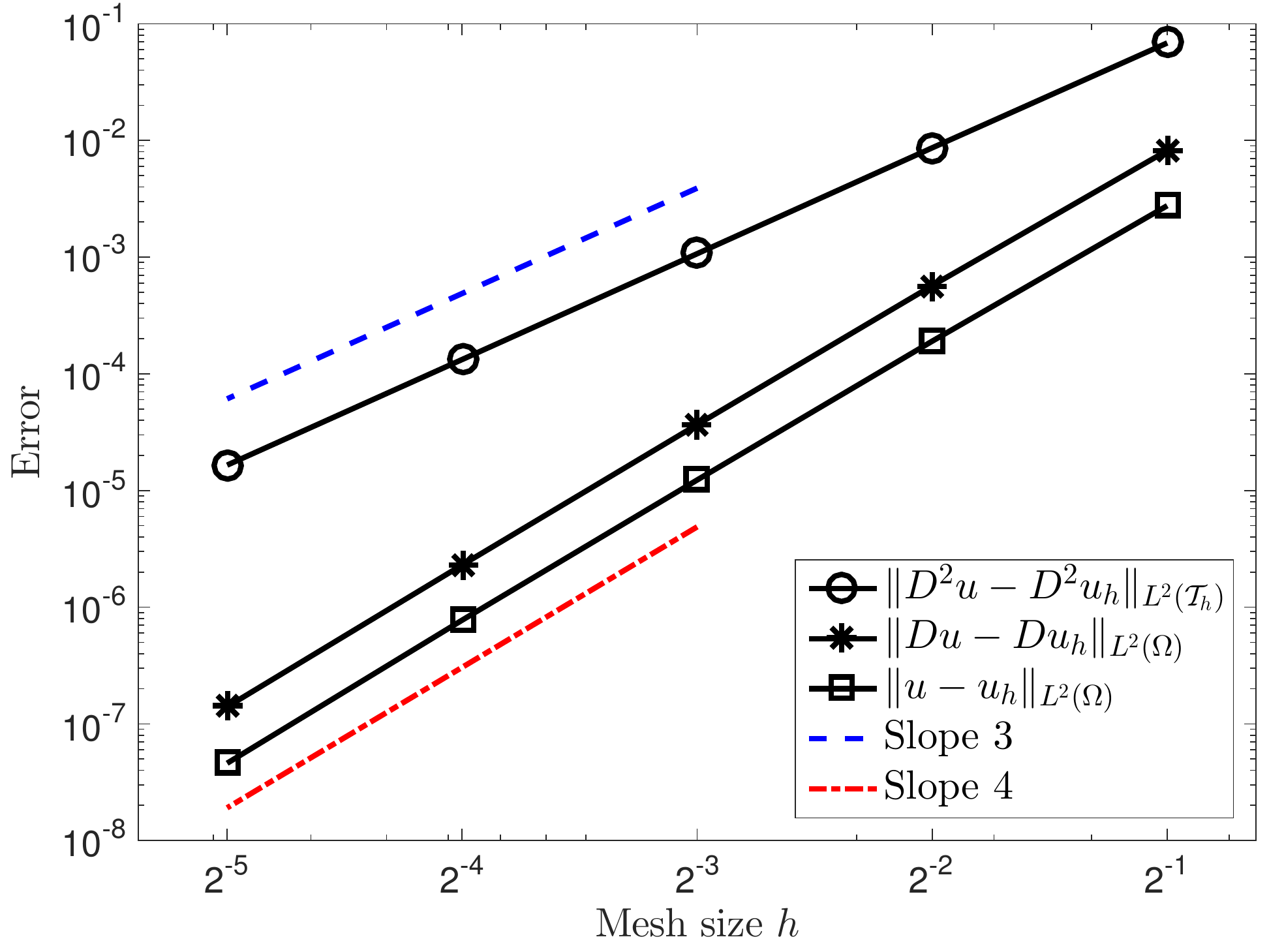}
  \label{fig:test1-k4}
} 
\caption{Convergence rate for the numerical scheme \eqref{eq:nondiv-h}
  applied to the linear elliptic equations in non-divergence form \eqref{eq:nondiv} for
    Experiment 1.}
\label{fig:test1}
\end{figure}

%%%%%%%%%%%%%%%%%%%%%%%%%%%%%%%%%%%%%%%%%%%%%%%%%%%%%%%%%%%%
%% Experiment 2 
%%%%%%%%%%%%%%%%%%%%%%%%%%%%%%%%%%%%%%%%%%%%%%%%%%%%%%%%%%%%
\subsection{Second experiment} \label{subsec:experiment2}
In this experiment, we consider the linear elliptic equations in non-divergence form with
lower-order terms on $\Omega = (-1,1)^2$: 
$$ 
A:D^2u + \bm{b}\cdot \nabla u - cu = f \quad \text{in }\Omega, \qquad 
u = 0 \quad\text{on }\partial\Omega.
$$ 
Here, $A$ is taken the same as \eqref{eq:nondiv-test1}, $\bm{b} =
(x_1, x_2)^T$, $c = 3$. By choosing $\lambda = 1$, we have 
$$ 
\frac{|A|^2 + |\bm{b}|^2/(2\lambda) + (c/\lambda)^2}{(\tr A +
    c/\alpha)^2} = \frac{19 + (x_1^2+x_2^2)/2}{49}\leq \frac{20}{49},
$$ 
which means that the Cordes condition holds for $\varepsilon = 9/20$
(see Remark \ref{rk:lower-order}).  The right hand side $f$ is chosen
so that the exact solution is \eqref{eq:nondiv-test1-u}. The scheme
converges with the optimal order $h^{k-1}$ in the broken $H^2$ norm,
as shown in Table \ref{tab:test2}.  The convergence orders in $H^1$
and $L^2$ norms are similar to the Experiment 1. 

\begin{table}[!htbp]
\centering 
\captionsetup{justification=centering}
\scriptsize
\begin{tabular}{cc|cc|cc|cc} 
\hline 
&$h$ & $\|u-u_h\|_{L^2(\Omega)}$ & Order & $|u -
u_h|_{H^1(\Omega)}$ & Order & $\|D^2u -
  D^2u\|_{L^2(\mathcal{T}_h)}$
& Order \\ \hline 
&$2^{-2}$ & 
1.72705E-03 & --- & 1.17301E-02 & --- & 1.41330E-01 & --- \\
&$2^{-3}$ &
4.10225E-04 & 2.07 & 2.33362E-03 & 2.33 & 3.59360E-02 & 1.98 \\
$k=3$ & $2^{-4}$ &
1.00457E-04 & 2.03 & 5.42524E-04 & 2.10 & 9.03321E-03 & 1.99 \\
&$2^{-5}$ &
2.49068E-05 & 2.01 & 1.33476E-04 & 2.02 & 2.26200E-03 & 2.00 \\
&$2^{-6}$ &
6.20697E-06 & 2.00 & 3.32792E-05 & 2.00 & 5.65735E-04 & 2.00 \\
\hline
&$2^{-1}$ & 
1.78055E-03 & --- & 6.63776E-03 & --- & 6.80847E-02 & --- \\  
&$2^{-2}$ &
1.21503E-04 & 3.87 & 4.62102E-04 & 3.84 & 8.63084E-03 & 2.98 \\
$k=4$ & $2^{-3}$ &
7.79999E-06 & 3.96 & 2.96137E-05 & 3.96 & 1.06983E-03 & 3.01 \\
&$2^{-4}$ &
4.88884E-07 & 4.00 & 1.85296E-06 & 4.00 & 1.32677E-04 & 3.01 \\
&$2^{-5}$ &
2.88593E-08 & 4.08 & 1.13437E-07 & 4.03 & 1.65056E-05 & 3.01 \\
\hline
\end{tabular}
\caption{Errors and observed convergence orders for Experiment 2.}
\label{tab:test2}
\end{table}

%%%%%%%%%%%%%%%%%%%%%%%%%%%%%%%%%%%%%%%%%%%%%%%%%%%%%%%%%%%%
%% Experiment 3 
%%%%%%%%%%%%%%%%%%%%%%%%%%%%%%%%%%%%%%%%%%%%%%%%%%%%%%%%%%%%
\subsection{Third experiment}
In this experiment, we solve the nonlinear HJB equations
\eqref{eq:HJB} in two dimensions on the domain $\Omega = (0,1)^2$.  
Following \cite{smears2014discontinuous}, we take $\Lambda = [0,
\pi/3] \times \mathrm{SO}(2)$, where $\mathrm{SO}(2)$ is the set of
$2\times 2$ rotation matrices. The coefficients are given by
$\bm{b}^\alpha = 0$, $c^\alpha = \pi^2$, and  
$$ 
A^\alpha = \frac{1}{2}\sigma^\alpha (\sigma^\alpha)^T, \qquad
\sigma^\alpha = R^T
\begin{pmatrix}
1 & \sin\theta \\
0 & \cos\theta
\end{pmatrix}, \qquad \alpha = (\theta, R) \in \Lambda.
$$ 
Since $\tr A^\alpha = 1$ and $|A^\alpha|^2 = (1+\sin^2\theta)/2 \leq
7/8$, the Cordes condition \eqref{eq:HJB-Cordes1} holds with
$\varepsilon = 1/7$ by taking $\lambda = 8\pi^2/7$.  We choose
$f^\alpha = \sqrt{3}\sin^2\theta / \pi^2 + g$, $g$ independent of
$\alpha$ so that the exact solution of the HJB equations
\eqref{eq:HJB} is $ u(x_1, x_2) = \exp(x_1x_2) \sin(\pi x_1) \sin(\pi
x_2)$.

On a sequence of uniform triangulations, we apply the numerical scheme
\eqref{eq:HJB-h} to the HJB equations \eqref{eq:HJB}. The finite element
spaces are defined by employing the 2D Hermite finite elements for
polynomial degrees $k=3$ and $k=4$. The plots of the errors given
in Figure \ref{fig:test3-k3-err} and \ref{fig:test3-k4-err} show that
the scheme converges with $\|u - u_h\|_{\lambda, h} =
\mathcal{O}(h^{k-1})$, which is in agreement with Theorem
\ref{tm:HJB-estimate}. The convergence orders in $H^1$ and $L^2$ norms
are similar to the Experiment 1 and Experiment 2. 

In the semismooth Newton algorithm, the initial guess is $u_h^0 = 0$,
and the stopping criterion is set to be $\|u_h^{j} - u_h^{j-1}\|_{\lambda, h} <
10^{-8}$. The convergence histories shown in Figure
\ref{fig:test3-k3-newton} and \ref{fig:test3-k3-newton} demonstrate
the fast convergence of the algorithm.  

\begin{figure}[!htbp]
\centering 
\captionsetup{justification=centering}
\subfloat[Convergence rate, $k=3$]{
  \includegraphics[width=0.48\textwidth]{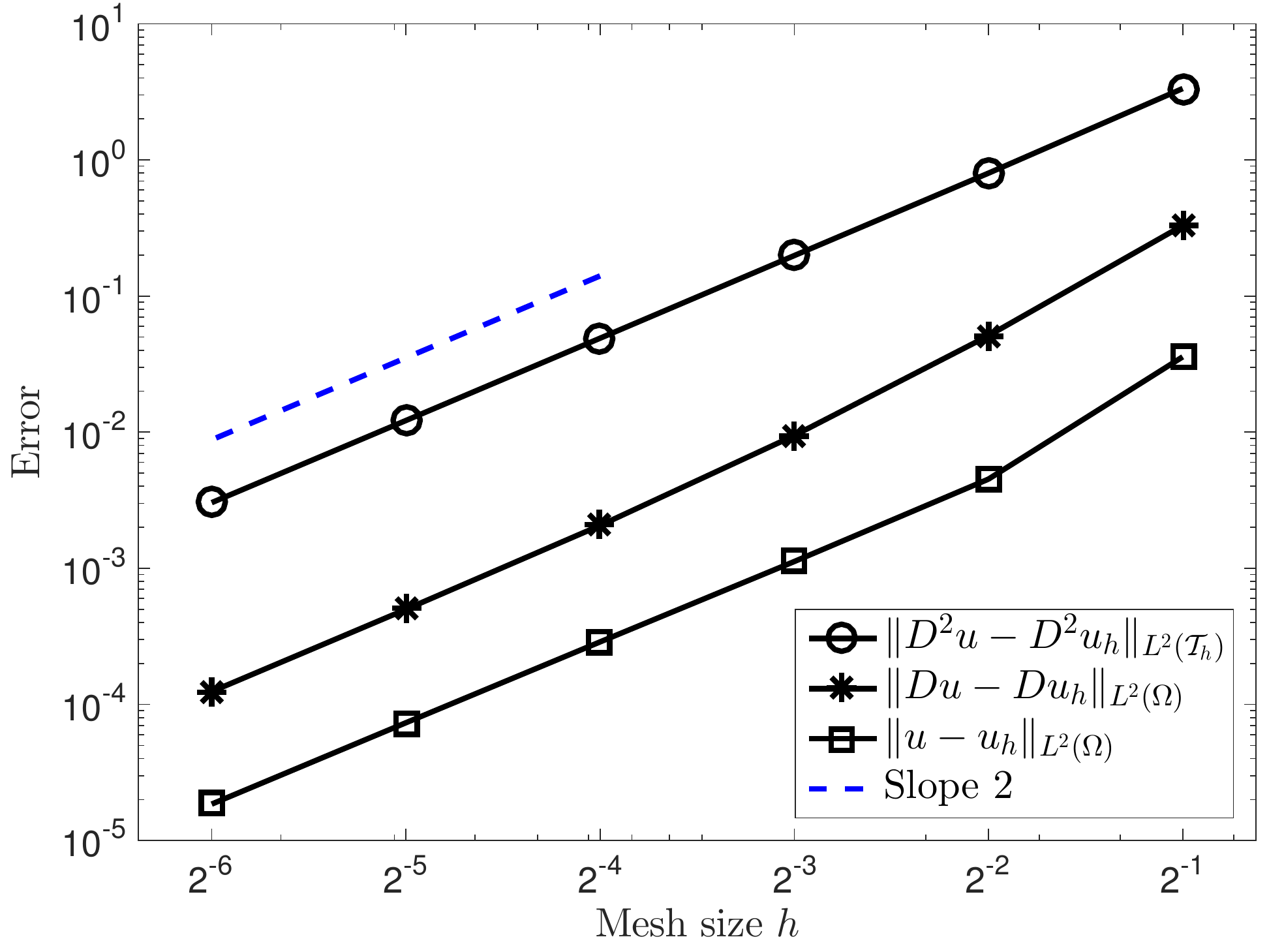}
  \label{fig:test3-k3-err}
}%
\subfloat[Semismooth Newton, $k=3$]{
  \includegraphics[width=0.48\textwidth]{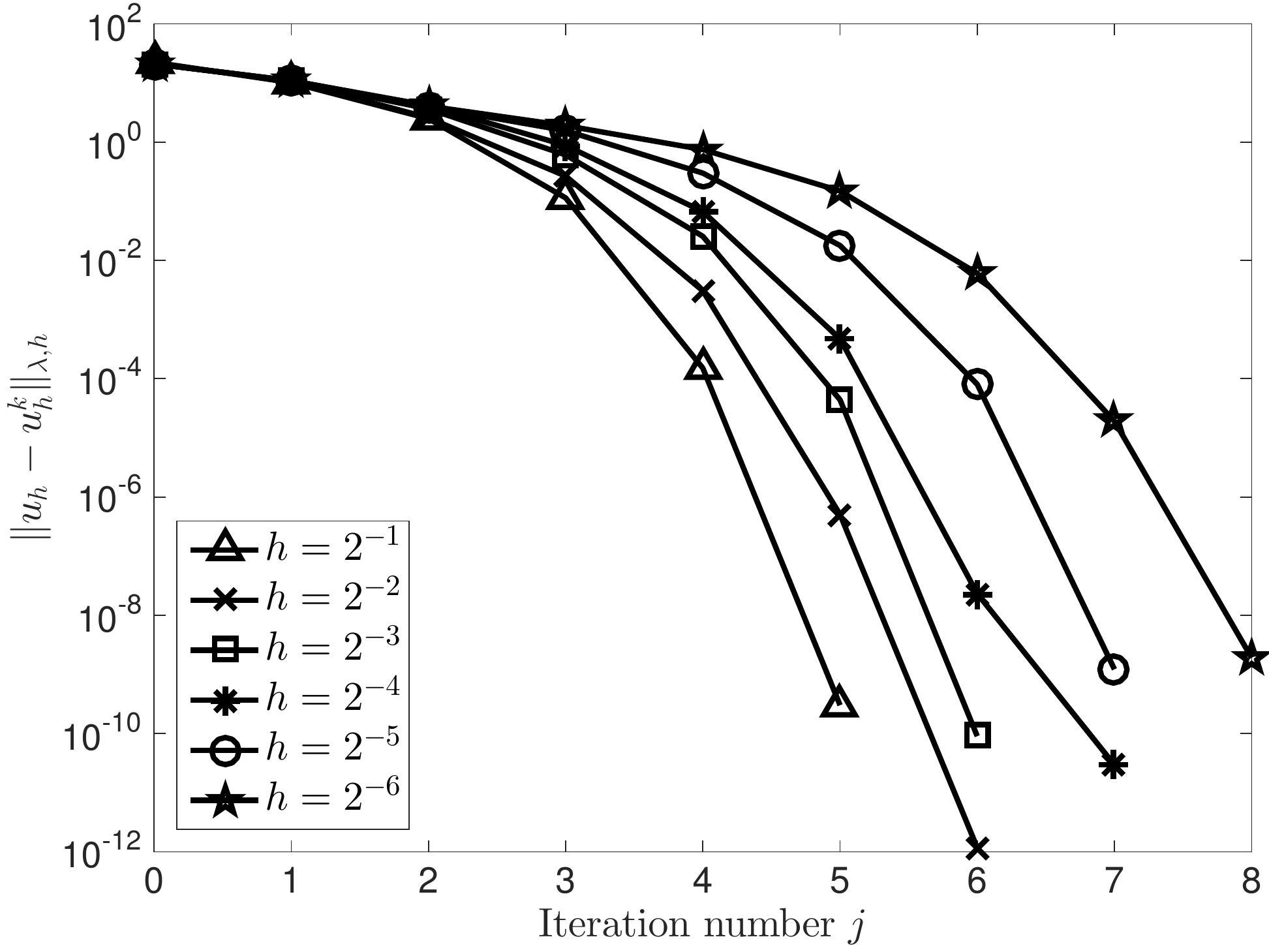}
  \label{fig:test3-k3-newton}
} \\
\subfloat[Convergence rate, $k=4$]{
  \includegraphics[width=0.48\textwidth]{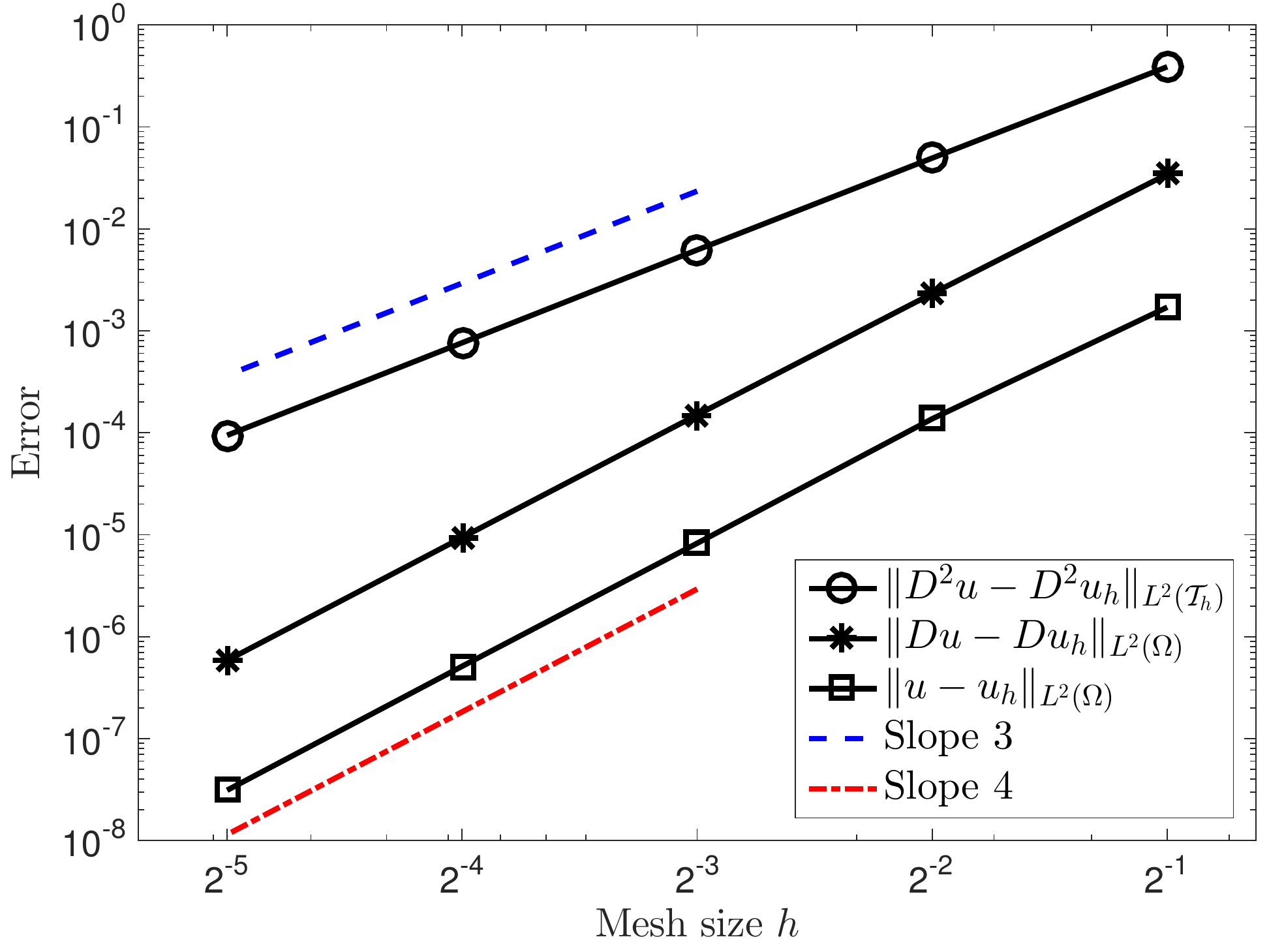}
  \label{fig:test3-k4-err}
} %
\subfloat[Semismooth Newton, $k=4$]{
  \includegraphics[width=0.48\textwidth]{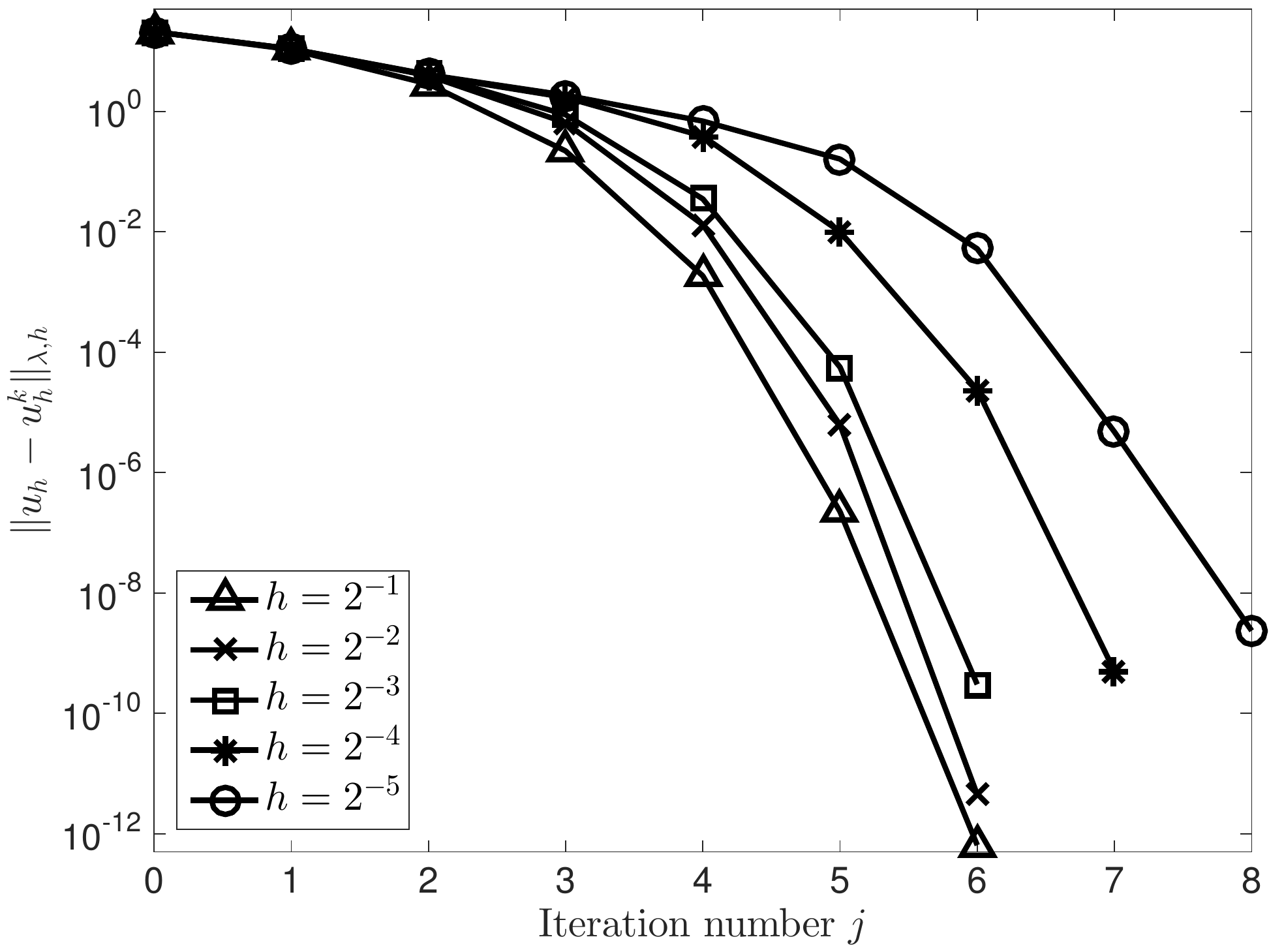}
  \label{fig:test3-k4-newton}
} 
\caption{Numerical scheme \eqref{eq:HJB-h} applied to the HJB equations
  \eqref{eq:HJB} for Experiment 3: Convergence rate and convergence
    histories of the semismooth Newton method.}
\label{fig:test3}
\end{figure}

%%%%%%%%%%%%%%%%%%%%%%%%%%%%%%%%%%%%%%%%%%%%%%%%%%%%%%%%%%%%
%% Experiment 4 
%%%%%%%%%%%%%%%%%%%%%%%%%%%%%%%%%%%%%%%%%%%%%%%%%%%%%%%%%%%%
\subsection{Fourth experiment}
In the last example, we consider a test case of the HJB equations
\eqref{eq:HJB} from \cite{smears2014discontinuous, gallistl2019mixed}
with near degenerate diffusion and a boundary layer in the solution.
Let $\Omega = (0,1)^2$, $\bm{b}^\alpha = (0,1)^T$, $c^\alpha = 10$,
and define 
$$ 
A^\alpha = \alpha^T
\begin{pmatrix}
20 & 1 \\
1 & 1/10
\end{pmatrix}
\alpha, \qquad \forall \alpha \in \Lambda:= \mathrm{SO}(2).
$$ 
For this choice of parameters and $\lambda = 1/2$, the Cordes
condition \eqref{eq:HJB-Cordes1} is satisfied for $\varepsilon =
0.0024$ (cf. \cite{smears2014discontinuous}). Let $\delta = 0.01$ and
$f^\alpha=A^\alpha:D^2u + \bm{b}^\alpha \cdot \nabla - c^\alpha u$,
where the exact solution is 
$$ 
u(x) = (2x_1 - 1)\big(\exp(1-|2x_1 - 1|) - 1 \big) \Big(x_2 +
    \frac{1-\exp(x_2/\delta)}{\exp(1/\delta) - 1} \Big).
$$ 
This solution exhibits a sharp boundary layer near the line
$\bar{\Omega} \cap \{x_2 = 1\}$. Following
\cite{smears2014discontinuous}, we use a sequence of graded bisection
meshes $\{\mathcal{T}_\ell\}_{\ell\in \mathbb{N}_0}$ with grading
factor $1/2$, see Figure \ref{fig:test4-mesh} for example.
More precisely, we mark the element $T$ whose $|T| >
C(x_{2,T}-1)^2/\#\mathcal{T}_\ell$, where $x_{2,T}$ is the second
component of barycenter of the element $T$, $\#\mathcal{T}_\ell$ is
the number of elements in $\mathcal{T}_\ell$. In the experiment, we
set $C = 120$, and use the 2D Hermite finite element spaces for
polynomial degrees $k=3$.  In Figure \ref{fig:test4-k3}, we plot the
errors in broken $H^2$-seminorm, $H^1$-seminorm and $L^2$ on the
graded bisection meshes, against the number of degrees of freedom
(ndof). We observe a reduction in the order of error
$\mathcal{O}(\text{ndof}^{-1})$, which confirms the second-order
convergence as shown in Theorem \ref{tm:HJB-estimate}. 

\begin{figure}[!htbp]
\centering 
\captionsetup{justification=centering}
\subfloat[Graded mesh at level 13: 3,587 vertices, 17,629 degrees of
freedom]{
  \includegraphics[width=0.43\textwidth]{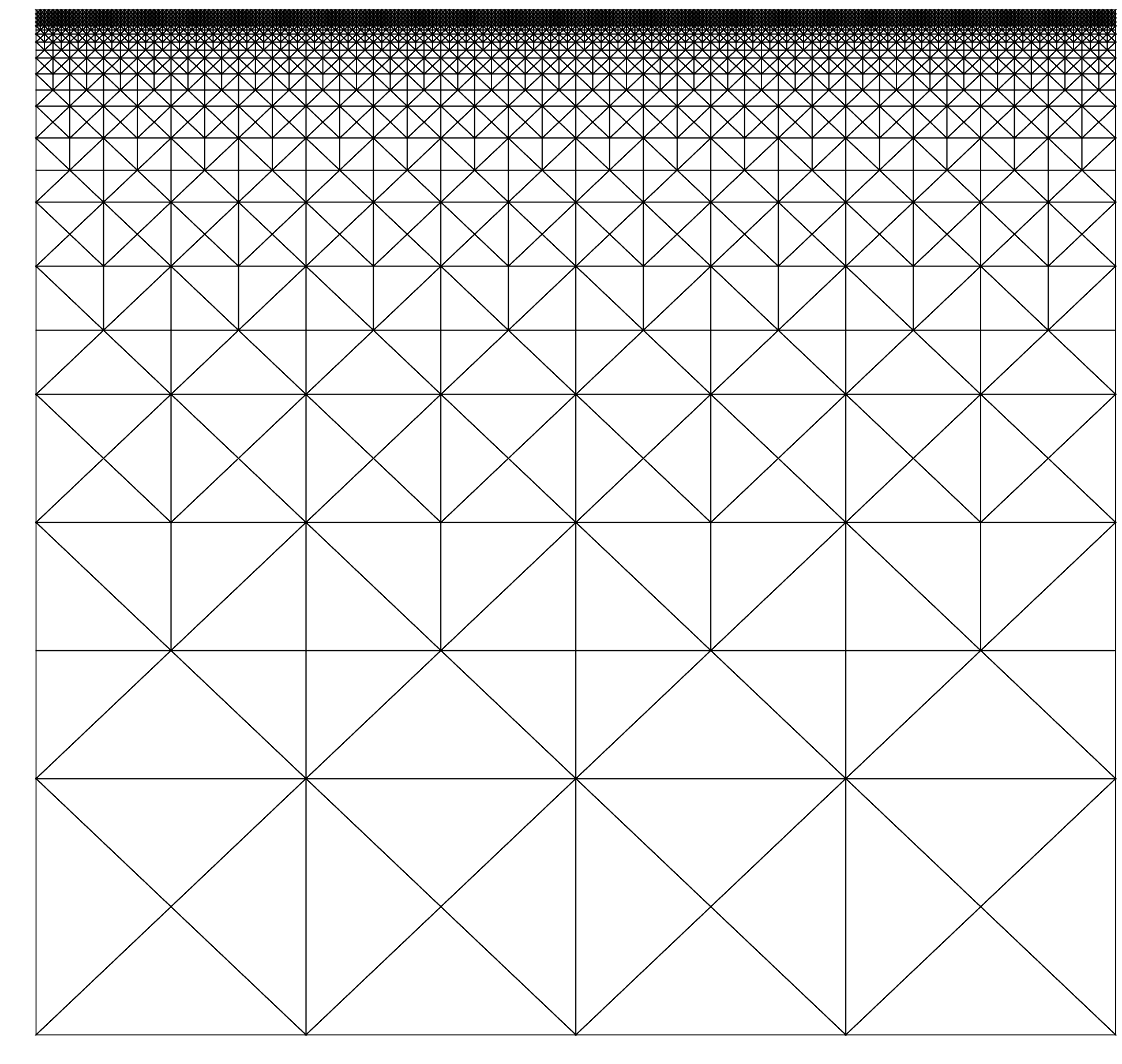}
  \label{fig:test4-mesh}
} %
\subfloat[Convergence rate, $k=3$]{
  \includegraphics[width=0.53\textwidth]{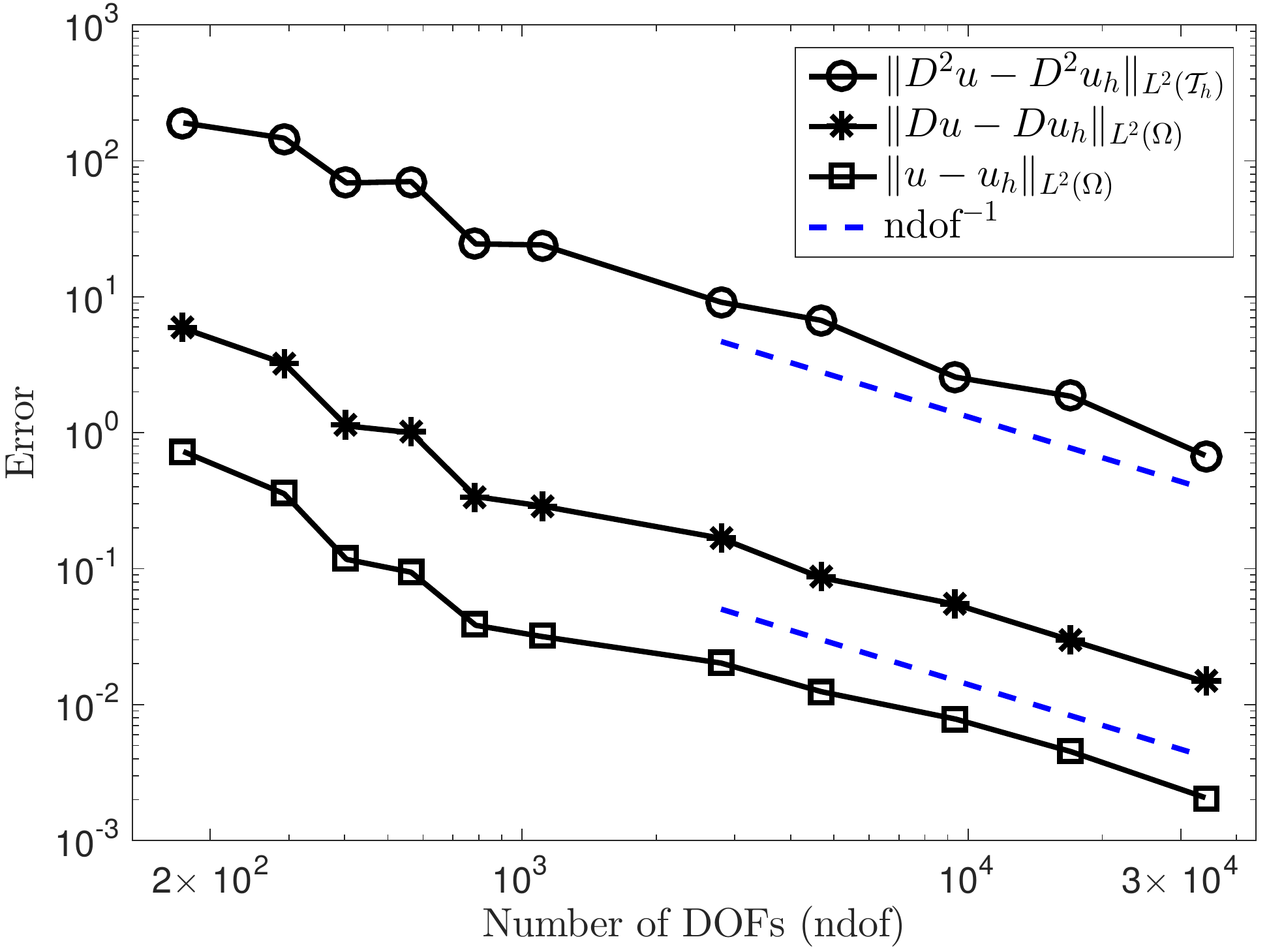}
  \label{fig:test4-k3}
} 
\caption{Convergence rate for the numerical scheme \eqref{eq:HJB-h}
  applied to the HJB equations \eqref{eq:HJB} for
    Experiment 4.}
\label{fig:test4}
\end{figure}

\section*{Acknowledgments}
The author would like to express his gratitude to Guangwei Gao and
Prof. Jun Hu in Peking University for their helpful discussions.

\bibliographystyle{siamplain}
\bibliography{Hermite.bib} 

\begin{thebibliography}{10}

\bibitem{barles1991convergence}
{\sc G.~Barles and P.~E. Souganidis}, {\em Convergence of approximation schemes
  for fully nonlinear second order equations}, Asymptotic Analysis, 4 (1991),
  pp.~271--283.

\bibitem{bonnans2003consistency}
{\sc J.~F. Bonnans and H.~Zidani}, {\em Consistency of generalized finite
  difference schemes for the stochastic {HJB} equation}, SIAM Journal on
  Numerical Analysis, 41 (2003), pp.~1008--1021.

\bibitem{brenner2007mathematical}
{\sc S.~Brenner and R.~Scott}, {\em The mathematical theory of finite element
  methods}, vol.~15, Springer Science \& Business Media, 2007.

\bibitem{caffarelli1997properties}
{\sc L.~A. Caffarelli and C.~E. Guti{\'e}rrez}, {\em Properties of the
  solutions of the linearized {Monge-Amp\`{e}re} equation}, American Journal of
  Mathematics, 119 (1997), pp.~423--465.

\bibitem{camilli2009finite}
{\sc F.~Camilli and E.~R. Jakobsen}, {\em A finite element like scheme for
  integro-partial differential {Hamilton-Jacobi-Bellman} equations}, SIAM
  Journal on Numerical Analysis, 47 (2009), pp.~2407--2431.

\bibitem{christiansen2018nodal}
{\sc S.~H. Christiansen, J.~Hu, and K.~Hu}, {\em Nodal finite element de {R}ham
  complexes}, Numerische Mathematik, 139 (2018), pp.~411--446.

\bibitem{ciarlet1978finite}
{\sc P.~G. Ciarlet}, {\em The finite element method for elliptic problems},
  North-Holland, 1978.

\bibitem{falk2013stokes}
{\sc R.~S. Falk and M.~Neilan}, {\em Stokes complexes and the construction of
  stable finite elements with pointwise mass conservation}, SIAM Journal on
  Numerical Analysis, 51 (2013), pp.~1308--1326.

\bibitem{feng2013recent}
{\sc X.~Feng, R.~Glowinski, and M.~Neilan}, {\em Recent developments in
  numerical methods for fully nonlinear second order partial differential
  equations}, SIAM Review, 55 (2013), pp.~205--267.

\bibitem{feng2017finite}
{\sc X.~Feng, L.~Hennings, and M.~Neilan}, {\em Finite element methods for
  second order linear elliptic partial differential equations in non-divergence
  form}, Mathematics of Computation, 86 (2017), pp.~2025--2051.

\bibitem{feng2019narrow}
{\sc X.~Feng and T.~Lewis}, {\em A narrow-stencil finite difference method for
  approximating viscosity solutions of fully nonlinear elliptic partial
  differential equations with applications to {Hamilton-Jacobi-Bellman}
  equations}, arXiv preprint arXiv:1907.10204,  (2019).

\bibitem{feng2018interior}
{\sc X.~Feng, M.~Neilan, and S.~Schnake}, {\em Interior penalty discontinuous
  {G}alerkin methods for second order linear non-divergence form elliptic
  {PDE}s}, Journal of Scientific Computing, 74 (2018), pp.~1651--1676.

\bibitem{fleming2006controlled}
{\sc W.~H. Fleming and H.~M. Soner}, {\em Controlled Markov processes and
  viscosity solutions}, vol.~25, Springer Science \& Business Media, 2006.

\bibitem{gallistl2017variational}
{\sc D.~Gallistl}, {\em Variational formulation and numerical analysis of
  linear elliptic equations in nondivergence form with {C}ordes coefficients},
  SIAM Journal on Numerical Analysis, 55 (2017), pp.~737--757.

\bibitem{gallistl2019mixed}
{\sc D.~Gallistl and E.~S\"{u}li}, {\em Mixed finite element approximation of
  the {H}amilton-{J}acobi-{B}ellman equation with {C}ordes coefficients}, SIAM
  Journal on Numerical Analysis, 57 (2019), pp.~592--614.

\bibitem{gilbarg2015elliptic}
{\sc D.~Gilbarg and N.~S. Trudinger}, {\em Elliptic partial differential
  equations of second order}, springer, 2015.

\bibitem{jensen2017finite}
{\sc M.~Jensen}, {\em {$L^2(H_\gamma^1)$} finite element convergence for
  degenerate isotropic {H}amilton-{J}acobi-{B}ellman equations}, IMA Journal of
  Numerical Analysis, 37 (2017), pp.~1300--1316.

\bibitem{jensen2013convergence}
{\sc M.~Jensen and I.~Smears}, {\em On the convergence of finite element
  methods for {H}amilton-{J}acobi-{B}ellman equations}, SIAM Journal on
  Numerical Analysis, 51 (2013), pp.~137--162.

\bibitem{kuratowski1965general}
{\sc K.~Kuratowski and C.~Ryll-Nardzewski}, {\em A general theorem on
  selectors}, Bull. Acad. Polon. Sci. S{\'e}r. Sci. Math. Astronom. Phys, 13
  (1965), pp.~397--403.

\bibitem{lakkis2011finite}
{\sc O.~Lakkis and T.~Pryer}, {\em A finite element method for second order
  nonvariational elliptic problems}, SIAM Journal on Scientific Computing, 33
  (2011), pp.~786--801.

\bibitem{li2019sequential}
{\sc R.~Li and F.~Yang}, {\em A sequential least squares method for elliptic
  equations in non-divergence form}, arXiv preprint arXiv:1906.03754,  (2019).

\bibitem{maugeri2000elliptic}
{\sc A.~Maugeri, D.~K. Palagachev, and L.~G. Softova}, {\em Elliptic and
  parabolic equations with discontinuous coefficients}, vol.~109, WILEY-VCH
  Verlag GmbH \& Co., 2000.

\bibitem{neilan2015discrete}
{\sc M.~Neilan}, {\em Discrete and conforming smooth de {R}ham complexes in
  three dimensions}, Mathematics of Computation, 84 (2015), pp.~2059--2081.

\bibitem{neilan2017numerical}
{\sc M.~Neilan, A.~J. Salgado, and W.~Zhang}, {\em Numerical analysis of
  strongly nonlinear {PDEs}}, Acta Numerica, 26 (2017), pp.~137--303.

\bibitem{neilan2019discrete}
{\sc M.~Neilan and M.~Wu}, {\em Discrete {M}iranda-{T}alenti estimates and
  applications to linear and nonlinear {PDE}s}, Journal of Computational and
  Applied Mathematics, 356 (2019), pp.~358--376.

\bibitem{nochetto2018discrete}
{\sc R.~H. Nochetto and W.~Zhang}, {\em Discrete {ABP} estimate and convergence
  rates for linear elliptic equations in non-divergence form}, Foundations of
  Computational Mathematics, 18 (2018), pp.~537--593.

\bibitem{qiu2019adaptive}
{\sc W.~Qiu and S.~Zhang}, {\em Adaptive first-order system least-squares
  finite element methods for second order elliptic equations in non-divergence
  form}, arXiv preprint arXiv:1906.11436,  (2019).

\bibitem{renardy2006introduction}
{\sc M.~Renardy and R.~C. Rogers}, {\em An introduction to partial differential
  equations}, vol.~13, Springer Science \& Business Media, 2006.

\bibitem{smears2013discontinuous}
{\sc I.~Smears and E.~S\"{u}li}, {\em Discontinuous {G}alerkin finite element
  approximation of nondivergence form elliptic equations with {C}ordes
  coefficients}, SIAM Journal on Numerical Analysis, 51 (2013), pp.~2088--2106.

\bibitem{smears2014discontinuous}
{\sc I.~Smears and E.~S\"{u}li}, {\em Discontinuous {G}alerkin finite element
  approximation of {H}amilton-{J}acobi-{B}ellman equations with {C}ordes
  coefficients}, SIAM Journal on Numerical Analysis, 52 (2014), pp.~993--1016.

\bibitem{smears2016discontinuous}
{\sc I.~Smears and E.~S{\"u}li}, {\em Discontinuous {G}alerkin finite element
  methods for time-dependent {H}amilton-{J}acobi-{B}ellman equations with
  {C}ordes coefficients}, Numerische Mathematik, 133 (2016), pp.~141--176.

\bibitem{ulbrich2002semismooth}
{\sc M.~Ulbrich}, {\em Semismooth {Newton} methods for operator equations in
  function spaces}, SIAM Journal on Optimization, 13 (2002), pp.~805--841.

\bibitem{wang2018primal}
{\sc C.~Wang and J.~Wang}, {\em A primal-dual weak {G}alerkin finite element
  method for second order elliptic equations in non-divergence form},
  Mathematics of Computation, 87 (2018), pp.~515--545.

\end{thebibliography}
\end{document}